\documentclass[12pt]{article}
\usepackage{amsfonts}
\usepackage{amsthm}
\usepackage{amsmath}
\usepackage{amssymb}
\usepackage{epsfig}
\usepackage{color}

\def\jmenovatel{{655978752}} 
\newtheorem{theorem}{Theorem}
\newtheorem{lemma}[theorem]{Lemma}
\newtheorem{observation}[theorem]{Observation}

\newenvironment{lml}[1]{{\noindent {\bf Lemma #1.\boldmath$a$\unboldmath}}\em}{\\ \medskip}
\newenvironment{lmlk}[1]{{\noindent {\bf Lemma #1.\boldmath$a.b$\unboldmath}}\em}{\\ \medskip}
\begin{document}
\title{A superlinear bound on the number of perfect matchings in cubic
  bridgeless graphs} \author{Louis Esperet\thanks{CNRS, Laboratoire
    G-SCOP, Grenoble, France.  E-mail: {\tt
      louis.esperet@g-scop.fr}. Partially supported by the European
    project \textsc{ist fet Aeolus}.}\and Franti{\v s}ek Kardo{\v
    s}\thanks{Institute of Mathematics, Faculty of Science, Pavol
    Jozef \v{S}af\'arik University, Ko\v{s}ice, Slovakia. E-mail: {\tt
      frantisek.kardos@upjs.sk}. Supported by Slovak Science
    and Technology Assistance Agency under contract
    No. APVV-0007-07.}\and Daniel Kr{\'a}l'\thanks{Institute for
    Theoretical Computer Science, Faculty of Mathematics and Physics,
    Charles University, Prague, Czech Republic.  E-mail: {\tt
      kral@kam.mff.cuni.cz}. Supported by The Czech Science Foundation under contract No. GACR 201/09/0197. The Institute for Theoretical Computer
    Science is supported by Ministry of Education of the Czech
    Republic as project 1M0545.}} \date{} \maketitle

\begin{abstract}
  Lov\'asz and Plummer conjectured in the 1970's that cubic bridgeless
  graphs have exponentially many perfect matchings. This conjecture
  has been verified for bipartite graphs by Voorhoeve in 1979, and for
  planar graphs by Chudnovsky and Seymour in 2008, but in general only
  linear bounds are known. In this paper, we provide the first
  superlinear bound in the general case.
\end{abstract}

\section{Introduction}\label{intro}

In this paper we study cubic graphs in which parallel edges are
allowed. A classical theorem of Petersen~\cite{bib-petersen1891}
asserts that every cubic bridgeless graph has a perfect
matching. In fact, it holds that every edge of a cubic bridgeless graph
is contained in a perfect matching. This implies that cubic bridgeless
graphs have at least 3 perfect matchings. In the 1970's, Lov{\'a}sz
and Plummer~\cite[Conjecture 8.1.8]{bib-lovasz86+} conjectured that
this quantity should grow exponentially with the number of vertices of a cubic bridgeless graph. The conjecture has
been verified in some special cases: Voorhoeve~\cite{bib-voorhoeve79}
proved in 1979 that $n$-vertex cubic bridgeless bipartite graphs have
at least $6\cdot\left(4/3\right)^{n/2-3}$ perfect matchings. Recently,
Chudnovsky and Seymour~\cite{bib-chudnovsky08+} proved that cubic
bridgeless planar graphs with $n$ vertices have at least
$2^{n/655978752}$ perfect matchings; Oum~\cite{bib-oum09} proved that cubic bridgeless claw-free graphs with $n$ vertices have at least $2^{n/12}$ perfect matchings.

However, in the general case all known bounds are linear. Edmonds, Lov\'asz, and Pulleyblank~\cite{bib-edmonds82+}, inspired by
Naddef~\cite{bib-naddef82}, proved in 1982 that the dimension of the
perfect matching polytope of a cubic bridgeless $n$-vertex graph is at
least $n/4+1$ which implies that these graphs have at least $n/4+2$
perfect matchings. The bound on the dimension of the perfect matching
polytope is best possible, but combining it with the study of the
brick and brace decomposition of cubic graphs yielded improved
bounds (on the number of perfect matchings in cubic bridgeless graphs)
of $n/2$~\cite{bib-previous}, and $3n/4-10$~\cite{bib-nous}.

Our aim in this paper is to show that the number of perfect matchings
in cubic bridgeless graphs is superlinear. More precisely, we
prove the following theorem:

\begin{theorem}\label{th:main}
For any $\alpha >0$ there exists a constant $\beta>0$ such that every
$n$-vertex cubic bridgeless graph has at least $\alpha n-\beta$ perfect
matchings.
\end{theorem}

\section{Notation}\label{sec:notation}

A graph $G$ is {\em $k$-vertex-connected} if $G$ has at least $k+1$
vertices, and remains connected after removing any set of at most
$k-1$ vertices. If $\{A,B\}$ is a partition of $V(G)$, the set
$E(A,B)$ of edges with one end in $A$ and the other in $B$ is called
an {\em edge-cut} or a {\it $k$-edge-cut} of $G$, where $k$ is the
size of $E(A,B)$. A graph is {\em $k$-edge-connected} if it has no
edge-cuts of size less than $k$. Finally, an edge-cut $E(A,B)$ is {\em
  cyclic} if the subgraphs induced by $A$ and $B$ both contain a
cycle. A graph $G$ is {\em cyclically $k$-edge-connected} if $G$ has
no cyclic edge-cuts of size less than $k$. The following is a usefull
observation that we implicitly use in our further considerations:

\begin{observation}\label{obs:cyc} If $G$ is a graph with minimum degree
three, in particular $G$ can be a cubic graph, then a $k$-edge-cut
$E(A,B)$ such that $|A|\ge k-1$ and $|B|\ge k-1$ must be cyclic.
\end{observation}

In particular, in a graph with minimum degree three, 2-edge-cuts are necessarily
cyclic. Hence, 3-edge-connected cubic graphs and cyclically
3-edge-connected cubic graphs are the same. 

We say that a graph $G$ is {\it $X$-near cubic} for a multiset $X$ of
positive integers, if the multiset of degrees of $G$ not equal to
three is $X$. For example, the graph obtained from a cubic graph by
removing an edge is $\{2,2\}$-near cubic.

If $v$ is a vertex of $G$, then $G\setminus v$ is the graph obtained
by removing the vertex $v$ together with all its incident edges. If $H$
is a connected subgraph of $G$, $G/H$ is the graph obtained by
contracting all the vertices of $H$ to a single vertex, removing
arising loops and preserving all parallel edges. 
Let $G$ and $H$ be two disjoint cubic graphs, $u$ a vertex of $G$
incident with three edges $e_1,e_2,e_3$, and $v$ a vertex of $H$
incident with three edges $f_1,f_2,f_3$. Consider the graph obtained
from the union of $G \setminus u$ and $H \setminus v$ by adding an
edge between the end-vertices of $e_i$ and $f_i$ ($1\le i \le3$)
distinct from $u$ and $v$. We say that this graph is obtained by
\emph{gluing} $G$ and $H$ through $u$ and $v$. Note that gluing a
graph $G$ and $K_4$ through a vertex $v$ of $G$ is the same as
replacing $v$ by a triangle.

A {\em Klee-graph} is inductively defined as being either $K_4$, or
the graph obtained from a Klee-graph by replacing a vertex by a
triangle. A {\em $b$-expansion} of a graph $G$, $b\ge 1$, is obtained
by gluing Klee-graphs with at most $b+1$ vertices each through some
vertices of $G$ (these vertices are then said to be \emph{expanded}).
For instance, a $3$-expansion of $G$ is a graph obtained by replacing
some of the vertices of $G$ with triangles, and by convention a
$1$-expansion is always the original graph. Observe that a
$b$-expansion of a graph on $n$ vertices has at most $bn$
vertices. Also observe that if we consider $k$ expanded vertices and
contract their corresponding Klee-graphs into single vertices in the
expansion, then the number of vertices decreases by at most $k(b-1)\le
kb$.

Let $G$ be a cyclically 4-edge-connected cubic graph and
$v_1v_2v_3v_4$ a path in $G$.  The graph obtained by {\em splitting
off} the path $v_1v_2v_3v_4$ is the graph obtained from $G$ by
removing the vertices $v_2$ and $v_3$ and adding the edges $v_1v_4$
and $v'_1v'_4$ where $v'_1$ is the neighbor of $v_2$ different from
$v_1$ and $v_3$, and $v'_4$ is the neighbor of $v_3$ different from
$v_2$ and $v_4$.

\begin{lemma}
\label{lm-splitoff}
Let $G$ be a cyclically $\ell$-edge-connected graph with at least
$2\ell+2$ vertices, let $G'$ be the graph obtained from $G$ by
splitting off a path $v_1v_2v_3v_4$, and let $v'_1$ be the neighbor of
$v_2$ different from $v_1$ and $v_3$, and $v'_4$ the neighbor of $v_3$
different from $v_2$ and $v_4$.  If $E(A',B')$ is a cyclic
$\ell'$-edge-cut of $G'$ with $\ell'<\ell$, then $\ell'\ge\ell-2$ and
neither the edge $v_1v_4$ nor the edge $v'_1v'_4$ is contained in the
cut $E(A',B')$.
\end{lemma}

\begin{proof}
Assume first that the edges $v_1v_4$ and $v_1'v_4'$ are both in the
cut $E(A',B')$. If $v_1,v_1 \in A'$ and $v_4,v_4'\in B'$ then
$E(A'\cup \{v_2\},B' \cup \{v_3\})$ is a cyclic $(\ell'-1)$-edge-cut
of $G$. Otherwise if $v_1,v_4' \in A'$ and $v_4,v_1'\in B'$ then
$E(A'\cup \{v_2,v_3\},B')$ is a cyclic $\ell'$-edge-cut of $G$. Since
$G$ is cyclically $\ell$-edge-connected, we can exclude these
cases. 

Assume now that only $v_1v_4$ is contained in the cut, i.e.,
$v_1\in A'$ and $v_4\in B'$ by symmetry. We can also assume by
symmetry that $v'_1$ and $v_4'$ are in $A'$.  However in this case,
the cut $E(A'\cup \{v_2,v_3\},B')$ is a cyclic $\ell'$-edge-cut of $G$
which is impossible.  Hence, neither $v_1v_4$ nor $v'_1v'_4$ is
contained in the cut.  Similarly, if $\{v_1,v'_1,v_4,v'_4\}\subseteq
A'$ or $\{v_1,v'_1,v_4,v'_4\}\subseteq B'$, then $G$ would contain a
cyclic $\ell'$-edge-cut.  

We conclude that it can be assumed that $\{v_1,v_4\}\subseteq A'$,
$\{v'_1,v'_4\}\subseteq B'$, and $|A'|\le |B'|$.  Say
$A:=A'\cup\{v_2,v_3\}$, $B:=B'$. Since $G'[A']$ has a cycle, $G[A]$
has a cycle, too. Since $|B|\ge \ell$, there is a cycle in $G[B]$ as
well. Therefore, $E(A,B)$ is a cyclic $(\ell'+2)$-edge-cut in $G$ and
thus $\ell'$ is either $\ell-2$ or $\ell-1$.
\end{proof}

A cubic graph $G$ is {\em $k$-almost cyclically $\ell$-edge-connected}
if there is a cyclically $\ell$-edge-connected cubic graph $G'$
obtained from $G$ by contracting sides of none, one or more cyclic
$3$-edge-cuts and the number of vertices of $G'$ is at least the
number of vertices of $G$ decreased by $k$. In particular, a graph $G$
is $2$-almost cyclically $4$-edge-connected graph if and only if $G$
is cyclically $4$-edge-connected or $G$ contains a triangle such that
the graph obtained from $G$ by replacing the triangle with a vertex is
cyclically $4$-edge-connected.  Observe that the perfect matchings of
the cyclically $4$-edge-connected cubic graph $G'$ correspond to
perfect matchings of $G$ (but several perfect matchings of $G$ can
correspond to the same perfect matching of $G'$ and some perfect
matchings of $G$ correspond to no perfect matching of $G'$).

We now list a certain number of facts related to perfect matchings in
graphs, that will be used several times in the proof. The first one,
due to Kotzig, concerns graphs (not necessarily cubic) with only one
perfect matching.

\begin{lemma}
\label{lm-bridge}
If $G$ is a graph with a unique perfect matching, then $G$ has a
bridge that is contained in the unique perfect matching of $G$.
\end{lemma}

A graph $G$ is said to be {\em matching-covered} if every edge is
contained in a perfect matching of $G$, and it is \emph{double
  covered} if every edge is contained in at least two perfect
matchings of $G$.

\begin{theorem}[\cite{bib-plesnik72}]
\label{thm-ef}
Every cubic bridgeless graph is matching-covered. Moreover, for any
two edges $e$ and $f$ of $G$, there is a perfect matching avoiding
both $e$ and $f$.
\end{theorem}

The following three theorems give lower bounds on the number of
perfect matchings in cubic graphs. 

\begin{theorem}[\cite{bib-chudnovsky08+}]
\label{thm-Klee}
Every planar cubic graph (and thus every Klee-graph) with $n$ vertices has
at least $2^{n/\jmenovatel}$ perfect matchings.
\end{theorem}

\begin{theorem}[\cite{bib-voorhoeve79}]
\label{thm-bip}
Every cubic bridgeless bipartite graph with $n$ vertices has at least
$(4/3)^{n/2}$ perfect matchings avoiding any given edge.
\end{theorem}

\begin{theorem}[\cite{bib-previous}]
\label{thm-half}
Every cubic bridgeless graph with $n$ vertices has at least $n/2$
perfect matchings.
\end{theorem}

The main idea in the proof of
Theorem~\ref{th:main} will be to cut the graph into pieces, apply
induction, and try to combine the perfect matchings in the different
parts. If they do not combine well then we will show that
Theorems~\ref{thm-Klee} and~\ref{thm-bip} can be applied to large
parts of the graphs to get the desired result. Typically this will
happen if some part is not double covered (some edge is in only one
perfect matching), or if no perfect matching contains a given edge
while excluding another one. In these cases the following
two lemmas will be very useful.

\begin{lemma}[\cite{bib-nous}]
\label{lm-double}
Every cyclically $3$-edge-connected cubic graph that is a not a Klee-graph
is double covered. In particular, every cyclically $4$-edge-connected cubic
graph is double covered.
\end{lemma}

\begin{lemma}
\label{lm-special}
Let $G$ be a cyclically $4$-edge-connected cubic graph and $e$ and $f$
two edges of $G$.  $G$ contains no perfect matchings avoiding $e$ and
containing $f$ if and only if the graph $G\setminus\{e,f\}$ is
bipartite and the end-vertices of $e$ are in one color class while
the end-vertices of $f$ are in the other.
\end{lemma}

\begin{proof}
Let $f=uv$, and assume that the graph $H$ obtained from $G$ by
removing the vertices $u,v$ and the egde $e$ has no perfect matching.
By Tutte's theorem, there exists a subset $S$ of vertices of $H$ such
that the number $k$ of odd components of $H \setminus S$ exceeds
$|S|$. Since $H$ has an even number of vertices, we actually have $k \ge
|S|+2$.  Let $S'=S \cup \{u,v\}$. The number of edges leaving $S'$ in
$G$ is at most $3|S'|-2$ because $u$ and $v$ are joined by an edge. On
the other hand, there are at least three edges leaving each odd
component of $H \setminus S$ with a possible exception for the (at
most two) components incident with $e$ (otherwise, we obtain a cyclic
2-edge-cut in $G$). Consequently, $k=|S|+2$, and
there are three edges leaving $|S|$ odd components and two edges
leaving the remaining two odd components.  Since $G$ is cyclically
4-edge-connected, all the odd components are single vertices and $G$ has
the desired structure.
\end{proof}

The key to prove Theorem~\ref{th:main} is to show by induction that
cyclically 4-edge-connected cubic graphs have a superlinear number of
perfect matchings \emph{avoiding any given edge}. In the proof we need
to pay special attention to 3-edge-connected graphs, because we were
unable to include them in the general induction process. The next
section, which might be of independent interest for the reader, will
be devoted to the proof of Lemma~\ref{lm-3conn}, stating that 3-edge-connected
cubic graphs have a linear number of perfect matchings avoiding any
given edge that is not contained in a cyclic 3-edge-cut (this
assumption on the edge cannot be dropped).

\section{$3$-edge-connected graphs}

We now introduce the \emph{brick
and brace decomposition} of matching-covered graphs (which will only
be used in this section).
For a simple graph $G$, we call
\emph{a multiple of $G$} any multigraph whose underlying simple
graph is isomorphic to $G$. 

An edge-cut $E(A,B)$ is {\em
  tight} if every perfect matching contains precisely one edge of
$E(A,B)$. If $G$ is a connected matching-covered graph with a tight
edge-cut $E(A,B)$, then $G[A]$ and $G[B]$ are also
connected. Moreover, every perfect matching of $G$ corresponds to a
pair of perfect matchings in the graphs $G/A$ and $G/B$. Hence, both
$G/A$ and $G/B$ are also matching-covered. We say that we have
decomposed $G$ into $G/A$ and $G/B$. If any of these graphs still have
a tight edge-cut, we can keep decomposing it until no graph in the
decomposition has a tight edge-cut. Matching-covered graphs without
tight edge-cuts are called \emph{braces} if they are bipartite and
\emph{bricks} otherwise, and the decomposition of a graph $G$ obtained
this way is known as the {\em brick and brace decomposition} of $G$.

Lov{\'a}sz~\cite{bib-lovasz87} showed that the collection of graphs
obtained from $G$ in any brick and brace decomposition is unique up to
the multiplicity of edges. This allows us to speak of \emph{the} brick
and brace decomposition of $G$, as well as \emph{the} number of bricks
(denoted $b(G)$) and \emph{the} number of braces in the decomposition
of $G$. The brick and brace decomposition has the following
interesting connection
with the number of perfect matchings:

\begin{theorem}[Edmonds \emph{et al.}, 1982]
\label{thm-bb}
If $G$ is a matching-covered $n$-vertex graph with $m$ edges, then $G$
has at least $m-n+1-b(G)$ perfect matchings.
\end{theorem}

A graph is said to be {\em bicritical} if $G\setminus\{u,v\}$ has a
perfect matching for any two vertices $u$ and $v$. Edmonds \emph{et
al.}~\cite{bib-edmonds82+} gave the following characterization of
bricks:

\begin{theorem}[Edmonds \emph{et al.}, 1982]
\label{th-brick}
A graph $G$ is a brick if and only if it is $3$-vertex-connected and
bicritical.
\end{theorem}

It can also be proved that a brace is a bipartite graph such that for
any two vertices $u$ and $u'$ from the same color class and any two
vertices $v$ and $v'$ from the other color class, the graph
$G\setminus\{u,u',v,v'\}$ has a perfect matching,
see~\cite{bib-lovasz86+}. We finish this brief introduction to the
brick and brace decomposition with two lemmas on the number of bricks
in some particular classes of graphs.

\begin{lemma}[see~\cite{bib-previous}]
\label{lm-bb-cubic}
If $G$ is an $n$-vertex cubic bridgeless
graph, then $b(G) \le n/4$.
\end{lemma}

\begin{lemma}[see~\cite{bib-nous}]
\label{lm-bb-bip}
If $G$ is a bipartite matching-covered graph, then $b(G)=0$.
\end{lemma}

We now show than any 3-edge-connected cubic graph $G$ has a linear number of
perfect matchings avoiding any edge $e$ not contained in a cyclic
3-edge-cut. We consider two cases: if $G-e$ is matching-covered, we
show that its decomposition contains few bricks
(Lemma~\ref{lm-bb-3e}). If $G-e$ is not matching-covered, we show that
for some edge $f$,  $G-\{e,f\}$ is matching-covered and contains few
bricks in its decomposition (Lemma~\ref{lm-bb-3ef}).

\begin{lemma}
\label{lm-bb-3e}
Let $G$ be a $3$-edge-connected cubic graph and $e$ an edge of $G$
that is not contained in a cyclic $3$-edge-cut of $G$.  If $G-e$ is
matching-covered, then the number of bricks in the brick and brace
decomposition of $G-e$ is at most $3n/8-2$.
\end{lemma}

\begin{proof}
Let $u$ and $v$ be the end-vertices of $e$. Clearly, the edges between
$\{u\}\cup N(u)$ and the other vertices and the edges between
$\{v\}\cup N(v)$ and the other vertices form tight edges-cuts in
$G-e$. Splitting along these tight edge-cuts, we obtain two multiples
of $C_4$ and a graph $G'$ with $n-4$ vertices.  Depending whether $u$
and $v$ are in a triangle in $G$, $G'$ is either a $\{4,4\}$-near
cubic graph or a $\{5\}$-near cubic graph.

We will now keep splitting $G'$ along tight edge-cuts until we obtain
bricks and braces only. We show that any graph $H$ obtained during splitting will
be 3-edge-connected and it will be either a bipartite graph, a cubic graph, a $\{4,4\}$-near cubic
graph or a $\{5\}$-near cubic graph. Moreover, the edge $e$ will
correspond in a $\{4,4\}$-near cubic graph to an edge joining the two
vertices of degree four in $H$, and it will correspond in a
$\{5\}$-near cubic graph to a loop incident with the vertex of degree
five. 

If $H$ is a $\{4,4\}$-near cubic graph and the two vertices $u$
and $v$ of degree four have two common neighbors that are adjacent, we
say that $H$ contains a \emph{4-diamond} with end-vertices $u$ and $v$
(see the first picture of Figure~\ref{fig:c3c} for the example of a
4-diamond with end-vertices $v$ and $w$). After we construct the
decomposition, we prove the following estimate on the number of bricks
in the brick and brace decomposition of $H$:

\smallskip
{\em Claim. Assume that $H$ is not a multiple of $K_4$, and
  that it has $n_H$ vertices. Then $b(H) \le \frac38\,n_H -
  1$ if $H$ is cubic, $b(H)\le \frac38\,n_H - \frac34$ if $H$ is a
  $\{5\}$-near cubic graph or a $\{4,4\}$-near cubic graph without
  4-diamond, and $b(H)\le \frac38\,n_H - \frac14$ if $H$ contains a
  4-diamond.}
\smallskip

Observe that the claim implies that if $H$ has no 4-diamond, $b(H) \le
\frac38\,n_H - \frac12$ regardless whether $H$ is a multiple of $K_4$
or not.  To simplify our exposition, we consider the construction of
the decomposition and after each step, we assume that we have verified
the claim on the number of bricks for the resulting graphs and verify
it for the original one.

Let $H$ be a graph obtained through splitting along tight edge-cuts,
initially $H=G'$. Observe that $G'$ is $3$-edge-connected since
the edge $e$ is not contained in any cyclic $3$-edge-cut of $G$.

If $H$ is bipartite, from Lemma \ref{lm-bb-bip} we get $b(H)=0$.

\begin{figure}[htbp]
\begin{center}
\includegraphics[scale=0.7]{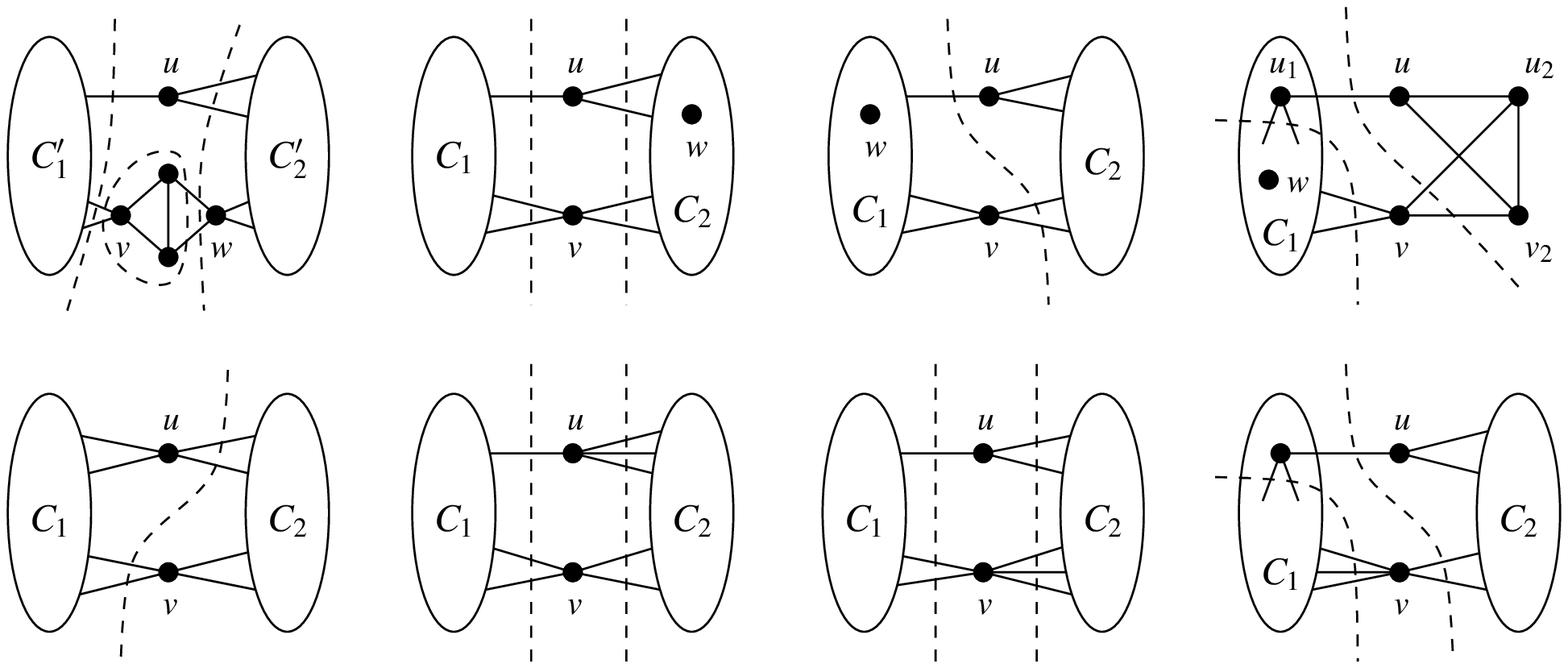}
\caption{Some cases if $H$ has a 2-vertex-cut $\{u,v\}$. The
tight edge-cuts are represented by dashed lines.\label{fig:c3c}}
\end{center}
\end{figure}

If $H$ is not bipartite, then by Theorem~\ref{th-brick} it is a brick
unless it is not $3$-vertex-connected or it is not bicritical. If it
is a brick then the inequalities of the claim are satisfied since
$n_H\ge 6$ unless $H$ is a multiple of $K_4$. Assume now that $H$ is
not $3$-vertex-connected. By the induction, the maximum degree of $H$
is at most five and since $H$ is $3$-edge-connected, it cannot contain
a cut-vertex.  Let $\{u,v\}$ be a $2$-vertex-cut of $H$. Since the sum
of the degrees of $u$ and $v$ is at most eight, the number of
components of $H\setminus\{u,v\}$ is at most two.  Let $C_1$ and $C_2$
be the two components of $H\setminus\{u,v\}$.  We now distinguish
several cases based on the degrees of $u$ and $v$ (symmetric
cases are omitted):
\begin{description}
\item[\boldmath$d_H(u)=d_H(v)=3$.\unboldmath]
     It is easy to verify that $H$ cannot be $3$-edge-connected.
\item[\boldmath$d_H(u)=3$, $d_H(v)=4$, $uv\in E(H)$.\unboldmath] It is again easy to verify that $H$ cannot
     be $3$-edge-connected.
\item[\boldmath$d_H(u)=3$, $d_H(v)=4$, $uv\notin E(H)$.\unboldmath]  
     Since $H$ is $3$-edge-connected,
     there must be exactly two edges between $v$ and each $C_i$,
     $i=1,2$.  By symmetry, we can assume that there a single edge
     between $u$ and $C_1$ and two edges between $u$ and $C_2$.
     
     Let $w$ be the other vertex of $H$ with degree four. Assume first
     that $v$ and $w$ are the end-points of a 4-diamond (this case is
     depicted in the first picture of Figure~\ref{fig:c3c}). Let
     $C_1'$ and $C_2'$ be the two components remaining in $H$ after
     removing $u$ and the four vertices of the 4-diamond. Without loss
     of generality, assume that the two neighbors of $v$ (resp. $w$)
     not in the diamond are in $C_1'$ (resp.~$C_2'$). We split $H$
     along the three following tight edge-cuts: the three edges
     leaving $C_1'$, the four edges leaving $C_2'\cup \{w\}$, and
     finally the four edges leaving $v$ and its two neighbors in the
     4-diamond.  We obtain a cubic graph $H_1$, a multiple of $K_4$, a
     multiple of $C_4$, and a $\{4,4\}$-near cubic graph $H_2$. If
     $n_1$ and $n_2$ are the orders of $H_1$ and $H_2$, we have
     $n_1+n_2=n_H-2$. By the induction, 
     $$b(H)\le 1+(\tfrac38\,n_1-\tfrac12)+(\tfrac38\,n_2-\tfrac14)\ =\
     \tfrac38\,n_H-\tfrac12 \ \le \ \tfrac38\,n_H-\tfrac14.$$

     Hence, we can assume that $H$ does not contain a 4-diamond. If
     $w$ is contained in $C_2$, both $C_1$ and $C_2$ have an odd
     number of vertices and both the cuts between $\{u,v\}$ and $C_i$,
     $i=1,2$, are tight (this case is depicted in the second picture
     of Figure~\ref{fig:c3c}).  After splitting along them, we obtain
     a multiple of $C_4$, a cubic graph of order $n_1$ and a
     $\{4,4\}$-near cubic graph of order $n_2$, such that
     $n_1+n_2=n_H$. By the induction, $$b(H) \le
     \tfrac38\,n_H-\tfrac12-\tfrac14= \tfrac38\,n_H-\tfrac34.$$

     It remains to analyse the case when $w$ is contained in $C_1$
     (this case is depicted in the third picture of
     Figure~\ref{fig:c3c}).  Splitting the graph along the tight
     edge-cut between $C_1\cup\{v\}$ and $C_2\cup\{u\}$, we obtain a
     $\{4,4\}$-near cubic graph $H_1$ which is not a multiple of $K_4$
     (otherwise $u$ would have more than one neighbor in $C_1$), and a
     cubic graph $H_2$. Observe that $H_1$ does not contain any
     4-diamond, since otherwise $H$ would contain one. If $H_2$ is not
     a multiple of $K_4$, then by the induction $b(H) \le
     \tfrac38\,(n_H+2)-\tfrac34-1 \le \tfrac38\,n_H-\tfrac34$. Assume
     now that $H_2$ is a multiple of $K_4$, and let $u_1$ be the
     neighbor of $u$ in $C_1$, and $u_2$ and $v_2$ be the vertices of
     $C_2$ (this case is depicted in the fourth picture of
     Figure~\ref{fig:c3c}). We split $H$ along the two following tight
     edge-cuts: the edges leaving $\{u,u_2,v_2\}$, and the edges
     leaving $\{u_1,u,v,u_2,v_2\}$. We obtain a multiple of $K_4$, a
     multiple of $C_4$, and a graph of order $n_H-4$, which is either
     a $\{4,4\}$-near cubic graph or a $\{5\}$-near cubic graph
     (depending whether $u_1=w$). In any case $$b(H) \le 1+
     \tfrac38\,(n_H-4)-\tfrac14 = \tfrac38\,n_H-\tfrac34.$$
\item[\boldmath$d_H(u)=d_H(v)=4$, $uv \in E(H)$.\unboldmath]
     The sizes of $C_1$ and $C_2$ must be even; otherwise, there is no
     perfect matching containing the edge $uv$. Hence, the number of
     edges between $\{u,v\}$ and $C_i$ is even and one of these cuts
     has size two, which is impossible since $H$ is $3$-edge-connected.
\item[\boldmath$d_H(u)=d_H(v)=4$, $uv\notin E(H)$.\unboldmath]
     Assume first that there are exactly two edges between each of the
     vertices $u$ and $v$ and each of the components $C_i$ (this case
     is depicted in the fifth picture of Figure~\ref{fig:c3c}). In
     this case, each $C_i$ must contain an even number of
     vertices. Hence, the edges between $C_1\cup\{u\}$ and
     $C_2\cup\{v\}$ form a tight edge-cut. Let $H_1$ and $H_2$ be the
     two graphs obtained by splitting along this tight
     edge-cut. Observe that if $H$ contains a 4-diamond, then at least
     one of $H_1$ and $H_2$ is a multiple of $K_4$. Moreover, $H_i$
     contains a 4-diamond if and only it is a multiple of
     $K_4$. Consequenty, if neither $H_1$ nor $H_2$ is a multiple of
     $K_4$, then none of $H$, $H_1$, and $H_2$ contains a
     4-diamond. Hence by the induction, $b(H) \le
     \tfrac38\,(n_H+2)-\tfrac34-\tfrac34 = \tfrac38\,n_H-\tfrac34$. If
     both $H_1$ and $H_2$ are multiples of $K_4$, then $H$ has $2 =
     \tfrac38 \times 6-\tfrac14$ bricks. Finally if exactly one of
     $H_1$ and $H_2$, say $H_2$, is a multiple of $K_4$, then $H$
     contains a 4-diamond and $H_1$ does not. Hence, $$b(H)\le
     1+\tfrac38\,(n_H-2)-\tfrac34 \le \tfrac38\,n_H-\tfrac14.$$

     If there are not exactly two edges between each of the vertices
     $u$ and $v$ and each of the components $C_i$, then we can assume
     that there is one edge between $u$ and $C_1$ and three edges
     between $v$ and $C_2$ (this case is depicted in the sixth picture
     of Figure~\ref{fig:c3c}).  Since $H$ is $3$-edge-connected, there
     are exactly two edges between $v$ and each of the components
     $C_i$, $i=1,2$.  Observe that each $C_i$ contains an odd number
     of vertices and thus the cuts between $C_i$ and $\{u,v\}$ are
     tight.  Splitting the graph along these tight edge-cuts, we
     obtain a multiple of $C_4$, a cubic graph $H_1$, and a
     $\{5\}$-near cubic graph $H_2$, of orders $n_1$ and $n_2$
     satisfying $n_1+n_2=n_H$. By the induction,
     $$b(H)\le \tfrac38\,n_1-\tfrac12+\tfrac38\,n_2-\tfrac12
     \le \tfrac38\,n_H-\tfrac34.$$
\item[\boldmath$d_H(u)=3$, $d_H(v)=5$, $uv\in E(H)$.\unboldmath]
     Since $H$ is $3$-edge-connected, the number of edges between $u$
     and each $C_i$ is one and between $v$ and each $C_i$ is
     two. Hence, both $C_1$ and $C_2$ contain an odd number of
     vertices and thus there is no perfect matching containing the
     edge $uv$ which is impossible since $H$ is
     matching-covered.
   \item[\boldmath$d_H(u)=3$, $d_H(v)=5$, $uv \notin
     E(H)$.\unboldmath] By symmetry, we can assume that there is one
     edge between $C_1$ and $u$ and two edges between $C_2$ and
     $v$. We have to distinguish two cases: there are either two or
     three edges between $C_1$ and $v$ (other cases are excluded by
     the fact that $H$ is $3$-edge-connected).

     If there are two edges between $C_1$ and $v$, the number of
     vertices of both $C_1$ and $C_2$ is odd (this case is depicted in
     the seventh picture of Figure~\ref{fig:c3c}). Hence, both the
     edge-cuts between $C_i$, $i=1,2$, and $\{u,v\}$ are tight.  The
     graphs obtained by splitting along these two edge-cuts are a
     multiple of $C_4$, a cubic graph $H_1$, and a $\{5\}$-near cubic
     graph $H_2$, of orders $n_1$ and $n_2$ satisfying
     $n_1+n_2=n_H$. By the induction,
     $$b(H)\le \tfrac38\,n_1-\tfrac12+\tfrac38\,n_2-\tfrac12
     \le \tfrac38\,n_H-\tfrac34.$$

     If there are three edges between $C_1$ and $v$, the edge-cut
     between $C_1\cup\{v\}$ and $C_2\cup\{u\}$ is tight (this case is
     depicted in the seventh and last picture of
     Figure~\ref{fig:c3c}). Splitting along this edge-cut, we obtain a
     $\{5\}$-near cubic graph $H_1$ which is not a multiple of $K_4$
     (the underlying simple graph has a vertex of degree two), and a
     cubic graph $H_2$, of orders $n_1$ and $n_2$ satisfying
     $n_1+n_2=n_H+2$. Let $u'$ be the new vertex of $H_1$ and let
     $u_1$ be its neighbor in $C_1$. Observe that the edges leaving
     $\{u',u_1,v\}$ form a tight edge-cut in $H_1$. Splitting along it
     we obtain a $\{5\}$-near cubic graph $H'_1$ of odred $n_1-2$ and
     a multiple of $C_4$. Hence, by induction,
     $$b(H)\le \tfrac38\,(n_1-2)-\tfrac12+\tfrac38\,n_2-\tfrac12
     \le \tfrac38\,n_H-\tfrac34.$$
\end{description}

It remains to analyse the case when $H$ is $3$-vertex-connected but
not bicritical.  Let $u$ and $u'$ be two vertices of $H$ such that
$H\setminus\{u,u'\}$ has no perfect matching.  Hence, there exists a
subset $S$ of vertices of $H$, $\{u,u'\}\subseteq S$, $|S|=k\ge 3$,
such that the number of odd components of $H\setminus S$ is at least
$k-1$. Since the order of $H$ is even, the number of odd components of
$H\setminus S$ is at least $k$. An argument based on counting the
degrees of vertices yields that there are exactly $k$ components of
$G\setminus S$; let $C_1,\ldots,C_k$ be these components. Clearly, for
each $i=1,\dots,k$ the cut between the component $C_i$ and the set $S$
is a tight edge-cut.  Let $H_i$ be the graph containing $C_i$ obtained
by splitting the cut; let $H_0$ be the graph containing vertices from
$S$ obtained after splitting all these cuts. Let $n_i$ be the order of
$H_i$. Clearly, $n_0=2k$ and $\sum_{i=1}^k n_i=n_H$.  An easy counting
argument yields that the number of edges joining $S$ and $H\setminus
S$ is between $3k$ and $3k+2$, hence, all the graphs $H_i$
($i=0,\dots,k$) are cubic, $\{4,4\}$-near cubic or $\{5\}$-near
cubic. However, at most two graphs $H_i$ are $\{4,4\}$-near cubic (or
one is $\{5\}$-near cubic).

If $H_0$ is bipartite, then $b(H_0)=0$ and applying the induction to
each $H_i$, we obtain that $H$ has at most
$$\tfrac38\,(n_1+\dots+n_k)-(k-2)\cdot\tfrac12-2\cdot\tfrac14= \tfrac38\,n_H-\tfrac12\,(k-1)
$$ bricks. Since $k \ge 3$, we have $b(H) \le
\tfrac38\,n_H-1$.

If $H_0$ is not bipartite, then all the $k$ tight edge-cuts are 3-edge-cuts, moreover, $H_0$ is a $\{4,4\}$-near cubic or
$\{5\}$-near cubic graph and all the graphs $H_i$ ($i=1,\dots,k$) are cubic. Applying the induction to each $H_i$ (including $H_0$), we obtain that $H$ has at most
$$\tfrac38\,(n_0+n_1+\dots+n_k)-k\cdot1-\tfrac14= \tfrac38\,n_H-\tfrac14\,(k+1)
$$ bricks. Since $k \ge 3$, we have $b(H) \le
\tfrac38\,n_H-1$, which finishes the proof of the claim.\\

As a consequence, using that $G'$ has $n-4$ vertices, we obtain that
the brick and brace decomposition of $G-e$ contains at most
$\tfrac38\,(n-4)-\tfrac14 = \tfrac38\,n-2$ bricks. Note that we made
sure troughout the proof, by induction, that all the graphs obtained
by splitting cuts are 3-edge-connected and are either bipartite,
cubic, $\{4,4\}$-near cubic, or $\{5\}$-near cubic.
\end{proof}

We now consider the case that $G-e$ is not matching-covered. Before
proving Lemma~\ref{lm-bb-3ef}, we will introduce the perfect matching
polytope of graphs.

The {\em perfect matching polytope} of a graph $G$ is the
convex hull of characteristic vectors of perfect matchings of $G$. The
sufficient and necessary conditions for a vector $w\in\mathbb{R}^{E(G)}$ to
lie in the perfect matching polytope are known~\cite{bib-edmonds65}:

\begin{theorem}[Edmonds, 1965]
\label{thm-polytope}
If $G$ is a graph, then a vector $w\in\mathbb{R}^{E(G)}$ lies in the perfect
matching polytope of $G$ if and only if the following holds:

\begin{itemize}
\item[(i)]$w$ is non-negative,

\item[(ii)]for every
vertex $v$ of $G$ the sum of the entries of $w$ corresponding to the
edges incident with $v$ is equal to one, and

\item[(iii)]for every set $S \subseteq V(G)$, $|S|$ odd,
  the sum of the entries corresponding to edges having exactly one
  vertex in $S$ is at least one.
\end{itemize}
\end{theorem}

\noindent It is also well-known that conditions (i) and (ii) are
necessary and sufficient for a vector to lie in the perfect matching
polytope of a bipartite graph $G$. \\

We now use these notions to prove the following result:

\begin{lemma}
\label{lm-bb-3ef}
Let $G$ be a $3$-edge-connected cubic graph $G$ and
$e$ an edge of $G$ such that $e$ is not contained in any cyclic
$3$-edge-cut of $G$. If $G- e$ is not matching-covered,
then there exists an edge $f$ of $G$ such that $G-\{e,f\}$
is matching-covered and the number of bricks in the brick and brace
decomposition of $G-\{e,f\}$ is at most $n/4-1$.
\end{lemma}

\begin{proof}
Since $G$ is not matching-covered, there exists an edge $f$ that
is contained in no perfect matching avoiding $e$. Since $G$ is matching
covered, $e$ and $f$ are vertex-disjoint.
Let $u$ and $u'$
be the end-vertices of $f$ and let $G'$ be the graph $G\setminus\{u,u'\}-e$.
By Tutte's theorem, there exists a subset $S'\subseteq V(G')$ such that
the number of odd components of the graph $G'\setminus S'$ is at least
$|S'|+1$. Since the number of vertices of $G'$ is even, the number
of odd components of $G'\setminus S$ is at least $|S'|+2$.

Let $S$ be the set $S'\cup\{u,u'\}$. The number of edges between $S$ and
$\overline{S}$ is at most $3|S|-2$ since the vertices $u$ and $u'$ are
joined by an edge. On the other hand, the number of edges
leaving $\overline{S}$ must be at least $3|S|-2$ since the graph $G$
is $3$-edge-connected and the equality can hold only if
the edge $e$ joins two different odd components of $(G-e)\setminus S$,
these two components have two additional edges leaving them and
all other components are odd components with exactly three edges leaving
them. Let $C_1$ and $C_2$ be the two components incident with $e$ and
let $C_3,\ldots,C_{|S|}$ be the other components. Since $e$ is contained
in no cyclic $3$-edge-cut of $G$, the components $C_1$ and $C_2$ are
single vertices.

Let $H$ be the graph obtained from $G-\{e,f\}$ by contracting the components
$C_3,\ldots,C_{|S|}$ to single vertices, and let $H_i$, $i=3,\ldots,|S|$
be the graph obtained from $C_i$ by introducing a new vertex incident
with the three edges leaving $C_i$. Each $H_i$, $i=3,\ldots,|S|$ is
matching-covered since it is a cubic bridgeless graph.
Since perfect matchings of $H_i$ combine with perfect matchings of $H$,
it is enough to show that the bipartite graph $H$
is matching-covered to establish that $G-\{e,f\}$ is matching-covered.

Observe that $H$ is $2$-edge-connected: otherwise, the bridge of $H$
together with $e$ and $f$ would form a cyclic $3$-edge-cut of $G$.

Let $v$ and $v'$ be the end-vertices of the edge $e$.  We construct an
auxiliary graph $H_0$ as follows: let $U$ and $V$ be the two color
classes of $H$, $U$ containing $u$ and $u'$ and $V$ containing $v$ and
$v'$.  Replace each edge of $H$ with a pair of edges, one directed
from $U$ to $V$ whose capacity is two and one directed from $V$ to $U$
whose capacity is one. In addition, introduce new vertices $u_0$ and
$v_0$.  Join $u_0$ to $u$ and $u'$ with directed edges of capacity two
and join $v$ and $v'$ to $v_0$ with directed edges of capacity two.

\begin{figure}[ht]
\begin{center}
\includegraphics[scale=1]{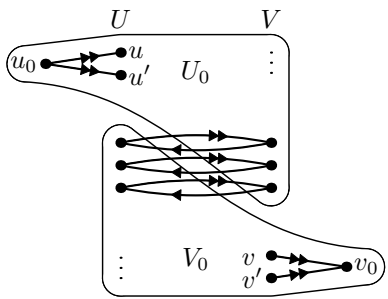}
\caption{A graph with no flow from $u_0$ to $v_0$ of order four.\label{fig:flow}}
\end{center}
\end{figure}
We claim that there exists a flow from $u_0$ to $v_0$ of order four.
If there is no such flow, the vertices of $H_0$ can be partitioned
into two parts $U_0$ and $V_0$, $u_0\in U_0$ and $v_0\in V_0$, such that
the sum of the capacities of the edges from $U_0$ to $V_0$ is at most three.
The fact that $H$ is $2$-edge-connected implies that
$\{u,u'\}\subseteq U_0$ and $\{v,v'\}\subseteq V_0$.
Hence, the number of edges between $U_0$ and $V_0$ must be at least three
since the edges between $U_0$ and $V_0$ correspond to an edge-cut in $G$.
Since the sum of the capacities of these edges is at most three,
all the three edges from $U_0$ to $V_0$ are directed from $V$ to $U$, see Figure \ref{fig:flow} for illustration.
However, the number of edges between $U\cap U_0$ and $V\cap U_0$ in $H$
is equal to $1$ modulo three based on counting incidences
with the vertices of $U\cap U_0$ and equal to $0$ modulo three based
on counting incident with vertices of $V\cap U_0$, which is impossible. This finishes
the proof of the existence of the flow.

Fix a flow from $u_0$ to $v_0$ of order four.
Let $ww'$ be an edge of $H$ with $w\in U$ and $w'\in V$.
Assign the edge $ww'$ weight of $1/3$, increase this weight
by $1/6$ for each unit of flow flowing from $w$ to $w'$ and
decrease by $1/6$ for each unit of flow from $w'$ to $w$.
Clearly, the final weight of $ww'$ is $1/6$, $1/3$, $1/2$ or $2/3$.
It is easy to verify that the sum of edges incident with each vertex
of $H$ is equal to one. In particular, the vector with entries
equal to the weights of the edges belongs to the perfect matching
polytope. Since all its entries are non-zero, the graph $H$
is matching-covered.

Let $n_i$ be the number of vertices of $C_i$, $i=3,\ldots,|S|$.
Since $H$ is bipartite, its brick and brace decomposition contains
no bricks by Lemma~\ref{lm-bb-bip}. The number of bricks in the brick and
brace decomposition of $C_i$ is at most $n_i/4$ by Lemma~\ref{lm-bb-cubic}.
Since $n_3+\ldots+n_{|S|}$ does not exceed $n-4$,
the number of bricks in the brick and brace decomposition of $G-\{e,f\}$
is at most $n/4-1$.
\end{proof}

\begin{lemma}
\label{lm-3conn}
Let $G$ be an $n$-vertex $3$-edge-connected cubic graph $G$
and $e$ an edge of $G$ that is not contained in any cyclic
$3$-edge-cut of $G$.  The number of perfect matchings of $G$ that
avoids $e$ is at least $n/8$.
\end{lemma}

\begin{proof}
If $G-e$ is matching-covered, then $b(G-e) \le 3n/8-2$ by
Lemma~\ref{lm-bb-3e}.  By Theorem~\ref{thm-bb}, the number of perfect
matchings of $G-e$ is at least
$$3n/2-1-n+1-(3n/8-2)=n/8+2\ge n/8\;\mbox{.}$$ If $G-e$ is not
matching-covered, then there exists an edge $f$ such that $G-\{e,f\}$
is matching-covered and the number of bricks in the brick and brace
decomposition of $G-\{e,f\}$ is at most $n/4-1$ by
Lemma~\ref{lm-bb-3ef}.  Theorem~\ref{thm-bb} now yields that the
number of perfect matchings of $G-\{e,f\}$ is at least
$$3n/2-2-n+1-(n/4-1)=n/4 \ge n/8\;\mbox{.}$$
\end{proof}

\section{Structure of the proof of Theorem~\ref{th:main}}

The proof is comprised by a series of lemmas -- they are referenced by
pairs X.$a$ or triples X.$a$.$b$, where
$\mbox{X}\in\{\mbox{A},\mbox{B},\mbox{C}, \mbox{D},\mbox{E}\}$ and
$a=0,1,\ldots$ and $b=1,2,\ldots$.  In the proof of Lemma~Y.$c$ or
Lemma~Y.$c$.$d$, we use Lemmas~X.$a$ and Lemmas~X.$a.b$ with either
$a<c$ or $a=c$ and X alphabetically preceeding Y.  The base of the
whole proof is thus formed by Lemmas~A.$0$, B.$0$, C.$0$.$b$, D.$0$.$b$ and
E.$0$.$b$, $b\in\{1,2,\dots\}$.

\medskip

\begin{lml}{A}
  There exists $\beta\ge 0$ such that any $3$-edge-connected
  $n$-vertex cubic graph $G$ contains at least $(a+3)n/24-\beta$
  perfect matchings.
\end{lml}

\begin{lml}{B}
There exists $\beta\ge 0$ such that any $n$-vertex bridgeless
cubic graph $G$ contains at least $(a+3)n/24-\beta$ perfect matchings.
\end{lml}

\begin{lmlk}{C}
There exists $\beta\ge 0$ such that for any cyclically $5$-edge-connected
cubic graph $G$ and any edge $e$ of $G$, the number of perfect matchings
of an arbitrary $b$-expansion of $G$ with $n$ vertices that avoid the edge $e$
is at least $(a+3)n/24-\beta$.
\end{lmlk}

\begin{lmlk}{D}
There exists $\beta\ge 0$ such that for any cyclically $4$-edge-connected
cubic graph $G$ and any edge $e$ of $G$ that is not contained in any cyclic
$4$-edge-cut of $G$, the number of perfect matchings
of an arbitrary $b$-expansion of $G$ with $n$ vertices that avoid the edge $e$
is at least $(a+3)n/24-\beta$.
\end{lmlk}

\begin{lmlk}{E}
There exists $\beta\ge 0$ such that for any cyclically $4$-edge-connected
cubic graph $G$ and any edge $e$ of $G$, the number of perfect matchings
of an arbitrary $b$-expansion of $G$ with $n$ vertices that avoid the edge $e$
is at least $(a+3)n/24-\beta$.
\end{lmlk}

The series A, B, C, D, and E of the lemmas will be proved in
Sections~\ref{section-A}, \ref{section-B}, \ref{section-C},
\ref{section-D}, and~\ref{section-E},
respectively. Section~\ref{sec:cutc4c} will be devoted to the study of
the connectivity of graphs obtained by cutting cyclically
4-edge-connected graphs into pieces.

\section{Proof of A-series of lemmas}\label{section-A}

\begin{proof}[Proof of Lemma A.$a$]
If $a=0$, the claim follows from Theorem~\ref{thm-half} with
$\beta=0$.  Assume that $a>0$.  Let $\beta_A$ be the constant from
Lemma A.$(a-1)$ and $\beta_E$ the constant from Lemma E.$(a-1)$.$b$,
where $b$ is the smallest integer such that
$$2^{b/\jmenovatel}\ge \tfrac{a+3}{24}\,b+3\;\mbox{.}$$ Let $\beta$ be the
smallest integer larger than $2\beta_A+12$ and $3\beta_E/2$ such that
$$2^{n/\jmenovatel}\ge \tfrac{a+3}{24}\,n-\beta$$
for every $n$.

We aim to prove with this choice of constants that any
$3$-edge-connected $n$-vertex cubic graph $G$ contains at least
$(a+3)n/24-\beta$ perfect matchings. Assume for the sake of
contradiction that this is not the case, and take $G$ to be a
counterexample with the minimum order.

If $G$ is cyclically $4$-edge-connected, then every edge of $G$ avoids
at least $(a+2)n/24-\beta_E$ perfect matchings by Lemma
E.$(a-1).b$. Hence, $G$ contains at least
$$\tfrac32\cdot \tfrac{a+2}{24}\,n -\tfrac32\,\beta_E\ge
\tfrac{a+3}{24}\,n-\tfrac32\,\beta_E$$ perfect matchings, as desired.

Let $G$ contain a cyclic $3$-edge-cut $E(A,B)$.
Let $e_i^A$ and $e_i^B$ ($i=1,2,3$) be the edges corresponding to the three edges of the cut $E(A,B)$ in $G/A$ and $G/B$, respectively; let $m_i^A$ ($m_i^B$) be the number of perfect matchings of $G/A$ ($G/B$) containing $e_i^A$ ($e_i^B$), $i=1,2,3$. 

If both $G/A$
and $G/B$ are double covered, apply Lemma A.$(a-1)$ to $G/A$ and
$G/B$. Let $n_A=|A|$ and $n_B=|B|$. Then $G/A$ and $G/B$ have
respectively at least
$$\tfrac{a+2}{24}\,(n_B+1)-\beta_A \,\mbox{ and }\, \tfrac{a+2}{24}\,
(n_A+1)-\beta_A$$ perfect matchings. 
Since $G/A$ and $G/B$ are double covered, $m_i^A\ge 2$ and $m_i^B\ge 2$ for $i=1,2,3$.
Hence, the number of perfect matchings of
$G$ is at least
%
$$
\gathered\sum_{i=1}^3m_i^A\cdot m_i^B \ge \sum_{i=1}^3 (2\cdot m_i^A + 2\cdot m_i^B - 4) = 2\cdot \sum_{i=1}^3 m_i^A + 2\cdot \sum_{i=1}^3 m_i^B-12 \ge
\\ \ge 2\cdot \tfrac{a+2}{24}\,n-2\beta_A-12\ge \tfrac{a+3}{24}n -2\beta_A-12
\;\mbox{.}
\endgathered
$$

Otherwise, Lemma~\ref{lm-double} implies that for every cyclic
$3$-edge-cut $E(A,B)$ at least one of the graphs $G/A$ and $G/B$ is a
Klee-graph. If both of them are Klee-graphs, then $G$ is a also a
Klee-graph and the bound follows from Theorem~\ref{thm-Klee} and the
choice of $\beta$.  Hence, exactly one of the graphs $G/A$ and $G/B$
is a Klee-graph.  Assume that there exists a cyclic $3$-edge-cut
$E(A,B)$ such that $G/A$ is a Klee-graph with more than $b$
vertices. Let $n_A=|A|$ and $n_B=|B|$. By the minimality of $G$, $G/B$
has at least $(a+3)(n_A+1)/24-\beta$ perfect matchings.  By the choice
of $b$, $G/A$ has at least $(a+3)(n_B+1)/24+3$ perfect matchings. 
Since $G/A$ and $G/B$ are matching covered, $m_i^A\ge 1$ and $m_i^B\ge 1$, $i=1,2,3$.
 The perfect
matchings of $G/A$ and $G/B$ combine to at least
$$\tfrac{a+3}{24}\,(n_B+1)+3+\tfrac{a+3}{24}\,(n_A+1)-\beta-3\ge
\tfrac{a+3}{24}\,n-\beta$$ perfect matchings of $G$.

We can now assume that for every cyclic 3-edge-cut $E(A,B)$ of $G$,
one of $G/A$ and $G/B$ is a Klee-graph of order at most $b$. In this
case, contract all the Klee sides of the cyclic $3$-edge-cuts.  The
resulting cubic graph $H$ is cyclically $4$-edge-connected and $G$ is
a $b$-expansion of $H$. By Lemma E.$(a-1).b$, $G$ has at least
$(a+2)n/24-\beta_E$ perfect matchings avoiding any edge present in
$H$.  Hence, $G$ contains at least
$$\tfrac32\cdot \tfrac{a+2}{24}\,n -\tfrac32\,\beta_E\ge
\tfrac{a+3}{24}\,n-\beta$$ perfect matchings, as desired.
\end{proof}

\section{Proof of B-series of lemmas}\label{section-B}

If $G$ is a cubic bridgeless graph, $E(A,B)$ a $2$-edge-cut with $A$
inclusion-wise minimal, then $G[A]$ is called a {\em semiblock} of
$G$. Observe that the semiblocks of $G$ are always vertex disjoint. If
$G$ has no $2$-edge-cuts, then it consists of a single semiblock
formed by $G$ itself. For a 2-edge-cut $E(A,B)$ of $G$, let $G_A$ ($G_B$) be
the graph obtained from $G[A]$ ($G[B]$) by adding an edge $f_A$ ($f_B$) between its two
vertices of degree two. Observe that if $G[A]$ is a semiblock, then
$s(G)=s(G_B)+1$, where the function $s$ assigns the number of
semiblocks.

\begin{lemma}
\label{lm-semiblock}
If $G$ is a cubic bridgeless graph, then any edge of $G$ is avoided
by at least $s(G)+1$ perfect matchings.
\end{lemma}

\begin{proof}
The proof proceeds by induction on the number of semiblocks of $G$.
If $s(G)=1$, then the statement is folklore. Assume $s(G)\ge 2$ and
fix an edge $e$ of $G$.  Let $E(A,B)$ be a $2$-edge-cut of $G$ such
that $G[A]$ is a semiblock.
If $e$ is contained in
$E(A,B)$ or in $G[B]$, then there are at least $s(G_B)+1=s(G)$ perfect
matchings avoiding $e$ in $G_B$ (if $e$ is in $E(A,B)$, avoiding the
edge $f_B$).  Choose among these perfect matchings one avoiding both $e$
and $f_B$.  This matching can be extended in two different ways to
$G[A]$ while the other matchings avoiding $e$ extend in at least one
way. Altogether, we have obtained $s(G)+1$ perfect matchings of $G$
avoiding $e$.

Assume that $e$ is inside $G[A]$. $G_B$ contains at least
$s(G_B)+1=s(G)$ perfect matchings avoiding $f_B$.  Each of them can be
combined with a perfect matching of $G_A$ avoiding $e$ and $f_A$ to
obtain a perfect matching of $G$ avoiding $e$. Moreover, a different
perfect matching of $G$ avoiding $e$ can be obtained by combining
a perfect matching of $G_B$ containing $f_B$ and a
perfect matching of $G_A$ avoiding $e$ and containing $f_A$ (if it
exists). If such a perfect matching does not exist, there must be another perfect matching of $G_A$ avoiding both $e$ and $f_A$. Since $s(G)\ge
2$ and $G_A$ has at least two perfect matchings avoiding $e$, we
obtain at least $s(G)+1$ perfect matchings of $G$ avoiding $e$.
\end{proof}

\begin{proof}[Proof of Lemma B.$a$]
Let $\beta_A$ be the constant from Lemma A.$a$ and set
$\beta=(\beta_A+2)^2$. First observe that Lemma A.$a$ implies that if
$G$ is a cubic bridgeless graph with $n$ vertices and $s$ semiblocks,
then $G$ has at least $(a+3)n/24-s(\beta_A+2)$ perfect matchings. This
can be proved by induction: if $s=1$ then $G$ is 3-edge-connected and
the result follows from Lemma A.$a$. Otherwise take a 2-edge-cut
$E(A,B)$ such that $G[A]$ is a semiblock of $G$, with $n_A=|A|$ and
$n_B=|B|$. Fix a pair of canonical perfect matchings of $G_A$, one containing $f_A$ and one avoiding it; fix another pair for $G_B$.
Each perfect matching of $G_A$ and each perfect matching of $G_B$ can be combined with a canonical perfect matching of the other part to a perfect matching of $G$. Since two combinations of the canonical perfect matchings are counted twice, we obtain at least
$$
\tfrac{a+3}{24}\,n_A-\beta_A+\tfrac{a+3}{24}\,n_B-(s-1)(\beta_A+2) - 2 =\tfrac{a+3}{24}\,n-s(\beta_A+2)
$$
perfect matchings of $G$, which
concludes the induction.

Consequently, if the number of semiblocks of $G$ is smaller than
$\beta_A+3$, the assertion of Lemma B.$a$ follows from Lemma A.$a$ by
the choice of $\beta$.

The rest of the proof proceeds by induction on the number of
semiblocks of $G$, under the assumption that $G$ has at least $\beta_A+3$
semiblocks. Let $E(A,B)$ be a $2$-edge-cut such that $G[A]$ is a
semiblock and let $n_A=|A|$ and $n_B=|B|$. By the induction, $G_B$ has
at least $(a+3)n_B/24-\beta$ perfect matchings, and by Lemma A.$a$,
$G_A$ has at least $(a+3)n_A/24-\beta_A$ perfect matchings. 
Let $m^A_f$ and $m^A_\varnothing$ ($m^B_f$ and $m^B_\varnothing$) be the number of perfect matchings in $G_A$ ($G_B$) containing and avoiding the edge $f_A$ ($f_B$). Clearly, $m^A_f$ and $m^B_f$ are non-zero, and $m^A_\varnothing\ge 2$; by Lemma \ref{lm-semiblock}, $m^B_\varnothing\ge s(G_B)+1\ge \beta_A+3$. Then $(m^A_\varnothing-2) \cdot (m^B_\varnothing-\beta_A-3)\ge 0$ and the number of perfect matchings of $G$ is at least
$$
\aligned
m^A_f\cdot m^B_f &+ m^A_\varnothing\cdot m^B_\varnothing \ge 
m^A_f+m^B_f-1 + (\beta_A+3)m^A_\varnothing + 2m^B_\varnothing - 2(\beta_A+3) \ge \\
&\ge m^A_f+m^A_\varnothing + m^B_f+m^B_\varnothing + (\beta_A+2)m^A_\varnothing + m^B_\varnothing - 2\beta_A-7 \\
&\ge \tfrac{a+3}{24}\,n_A-\beta_A+\tfrac{a+3}{24}\,n_B-
\beta+2(\beta_A+2)+(\beta_A+3)-2\beta_A-7  \\
&\ge \tfrac{a+3}{24}\,n-\beta\;\mbox{.}
\endaligned
$$
\end{proof}

\section{Proof of C-series of lemmas}\label{section-C}

Given an edge $e$ in a cyclically 5-edge-connected cubic graph $G$, there are several possible paths that can be split in such a way that perfect matchings of the reduced graph $H$ avoiding an edge correspond to perfect matchings of $G$ avoiding $e$. In the following two lemmas we prove that at least three of four such graphs $H$ are 4-almost cyclically 4-edge-connected.

\begin{lemma}
\label{lm-split-5-same}
Let $G$ be a cyclically $5$-edge-connected cubic graph with at least
$12$ vertices and let $v_1v_2v_3v_4$ be a path in $G$. Let $v'_4$ be
the neighbor of $v_3$ different from $v_2$ and $v_4$. At least one of
the graphs $H$ and $H'$ obtained from $G$ by splitting off the paths
$v_1v_2v_3v_4$ and $v_1v_2v_3v'_4$, respectively, is $4$-almost
cyclically $4$-edge-connected.
\end{lemma}

\begin{proof}
Let $v'_1$ be the neighbor of $v_2$ different from $v_1$ and $v_3$.
By Lemma \ref{lm-splitoff}, both $H$ and $H'$ are cyclically 3-edge-connected. Assume that neither $H$ nor $H'$ is cyclically $4$-edge-connected. 
i.e., $H$ contains a cyclic $3$-edge-cut $E(A,B)$ and $H'$ contains a
cyclic $3$-edge-cut $E(A',B')$. 
 By Lemma~\ref{lm-splitoff}, we can
assume by symmetry that $v_1\in A\cap A'$, $v'_1\in B\cap B'$, $v_4\in
A\cap B'$ and $v'_4\in A'\cap B$, see Figure \ref{fig:5ec}.

\begin{figure}[htbp]
\begin{center}
\includegraphics[scale=1]{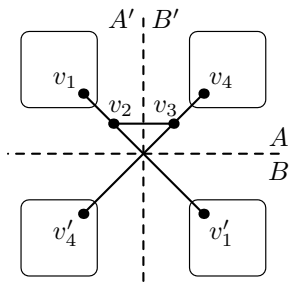}
\caption{After splitting off the paths $v_1v_2v_3v_4$ and $v_1v_2v_3v'_4$ in $G$ we obtain cyclic edge-cuts $E(A,B)$ and $E(A',B')$ in $H$ and $H'$, respectively. \label{fig:5ec}}
\end{center}
\end{figure}

We first show that at least one of $G[A]$, $G[A']$, $G[B]$ and $G[B']$
is a triangle. Assume that this is not the case. Hence, each of $A$,
$A'$, $B$ and $B'$ contains at least four vertices. Let $d(X)$ be the
number of edges leaving a vertex set $X$ in $G$.

Assume first that $|A\cap A'|=1$. Since $|A|\ge 4$ and $|A'|\ge 4$, it
follows that $|A\cap B'|\ge 3$ and $|A'\cap B|\ge 3$.  Since $G$ is
cyclically $5$-edge-connected, then $d(A\cap B')\ge 5$ and $d(A'\cap
B)\ge 5$.  As there is exactly one edge from $A\cap B'$ and one edge
from $A'\cap B$ leading to $\{v_2,v_3\}$, $|E(A,B)|+|E(A',B')|\ge
d(A\cap B')+d(B\cap A')-2\ge 8$ which is a contradiction.

We conclude that $|A\cap A'|\ge 2$ and, by symmetry, $|A\cap B'|\ge
2$, $|A'\cap B|\ge 2$ and $|B\cap B'|\ge 2$.  Since $G$ is cyclically
$5$-edge-connected, we have $d(X\cap Y)\ge 4$ for each $(X,Y) \in
\{A,B\} \times \{A',B'\}$ (with equality if and only if $X\cap Y$
consists of two adjacent vertices). As from each of the four sets
$X\cap Y$, there is a single edge going to $\{v_2,v_3\}$, we obtain
that the sum of $d(X\cap Y)$, $(X,Y) \in \{A,B\} \times \{A',B'\}$, is
at most $2|E(A,B)|+2|E(A',B')| +4=16$. Hence, all four sets $X \cap Y$
consist of two adjacent vertices and there are no edges between $A\cap
A'$ and $B\cap B'$, and between $A\cap B'$ and $A'\cap B$. In this
case $G$ must contain a cycle of length 3 or 4. Since $G$ has at least
8 vertices, it would imply that $G$ contains a cyclic edge-cut of size
at most four, a contradiction.

We have shown that for any cyclic 3-edge-cuts $E(A,B)$ in $H$ and
$E(A',B')$ in $H'$, at least one of the graphs $G[A]$, $G[B]$, $G[A']$
and $G[B']$ is a triangle. This implies that in $H$ or $H'$, say $H$,
all cyclic 3-edge-cuts $E(A,B)$ are such that $G[A]$ or $G[B]$ is
a triangle. The only way a triangle can appear is that there is a common neighbor of one of the
vertices $v_1$ and $v'_1$ and one of the vertices $v_4$ and $v'_4$. Since $G$ is cyclically $5$-edge-connected,
any pair of such vertices have at most one common neighbor (otherwise,
$G$ would contain a 4-cycle). In particular, $H$ has at most two
triangles and it is $4$-almost cyclically $4$-edge-connected.
\end{proof}

\begin{lemma}
\label{lm-split-5-diff}
Let $G$ be a cyclically $5$-edge-connected cubic graph with at least
$12$ vertices and let $v_1v_2v_3v_4$ and $v_1v_2v'_3v'_4$ be paths in
$G$ with $v_3\ne v'_3$.  At least one of the graphs $H$ and $H'$
obtained from $G$ by splitting off the paths $v_1v_2v_3v_4$ and
$v_1v_2v'_3v'_4$, respectively, is $4$-almost cyclically
$4$-edge-connected.
\end{lemma}

\begin{proof}
Let $v_5$ be the neighbor of $v_3$ different from $v_2$ and $v_4$, and
let $v'_5$ be the neighbor of $v'_3$ different from $v_2$ and $v'_4$.
Again, by Lemma \ref{lm-splitoff}, both $H$ and $H'$ are cyclically 3-edge-connected. We assume that neither $H$ nor $H'$ is cyclically
$4$-edge-connected and consider cyclic $3$-edge-cuts $E(A,B)$ of $H$
and $E(A',B')$ of $H'$.  For the sake of contradiction, assume that
each of $A$, $B$, $A'$, and $B'$ has the size at least four. By
Lemma~\ref{lm-splitoff}, we can also assume that $\{v_1,v_4\}\subseteq
A$ and $\{v'_3,v_5\}\subseteq B$. We claim that both $v'_4$ and $v'_5$
also belong to $B$. Clearly, at least one of them does (otherwise, $G$
would contain a cyclic $2$-edge-cut, which is impossible by
Lemma~\ref{lm-splitoff}).  Say, $v'_5$ does and $v'_4$ does not.  Let
$C=A\cup\{v_2,v_3,v'_3\}$ and $D=B\setminus\{v'_3\}$. The set $D$
contains the vertices $v_5$ and $v'_5$, which are distinct since $G$
has no 4-cycle. The edge-cut $E(C,D)$ is a $4$-edge-cut in $G$, and
since $G$ is cyclically $5$-edge-connected, we have $D=\{v_5,v'_5\}$
and thus $G[B]$ is a triangle as desired. A symmetric argument applies
if $v'_4$ is contained in $B$ and $v'_5$ is not.  We conclude that we
can restrict our attention without loss of generality to the following
case: $\{v_1,v_4\}\subseteq A$, $\{v'_3,v_5,v'_4,v'_5\}\subseteq B$,
$\{v_1,v'_4\}\subseteq A'$ and $\{v_3,v'_5,v_4,v_5\}\subseteq B'$, see Figure \ref{fig:5ec2}.

\begin{figure}[htbp]
\begin{center}
\includegraphics[scale=1]{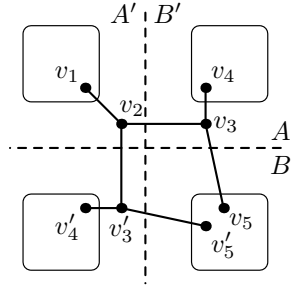}
\caption{After splitting off the paths $v_1v_2v_3v_4$ and $v_1v_2v'_3v'_4$ in $G$ we obtain cyclic edge-cuts $E(A,B)$ and $E(A',B')$ in $H$ and $H'$, respectively. \label{fig:5ec2}}
\end{center}
\end{figure}

As a consequence, $|X\cap Y|\ge 1$ for each $(X,Y) \in \{A,B\} \times
\{A',B'\}$ and the set $B\cap B'$ contains at least two vertices
($v_5$ and $v'_5$).  If $|A\cap A'|=1$, then both $|A\cap B'|$ and
$|B\cap A'|$ are at least three.  Consequently, $d(A\cap B')\ge 5$ and
$d(A'\cap B)\ge 5$ where $d(X)$ is the number of edges leaving $X$ in
$G$.  Since $|E(A,B)|+|E(A',B')|\ge d(A\cap B')+d(A'\cap B)-2\ge 8$,
this case cannot happen. Similarly, we obtain a contradiction if
$|A\cap B'|=1$ by inferring that $|E(A,B)|+|E(A',B')|\ge d(A\cap
A')+d(B\cap B')-3\ge 7$.  Hence, each of the numbers $d(X\cap Y)$,
$(X,Y) \in \{A,B\} \times \{A',B'\}$, is at least four (with equality
if and only if $X \cap Y$ consists of two adjacent vertices) and their
sum is at least $16$. Since exactly five edges leave the sets $X\cap
Y$ to $\{v_2,v_3,v'_3\}$, we obtain that the sum of $d(X\cap Y)$,
$(X,Y) \in \{A,B\} \times \{A',B'\}$, is at most
$2|E(A,B)|+2|E(A',B')| +5=17$. As a consequence, three of the sets $X
\cap Y$ consists of two adjacent vertices and there are no edges
between $A\cap A'$ and $B\cap B'$, and between $A\cap B'$ and $A'\cap
B$. In this case $G$ must contain a cycle of length 3 or 4. Since $G$
has at least 8 vertices, it would imply that it contains a cyclic
edge-cut of size at most four, a contradiction.

We proved that for any cyclic 3-edge-cuts $E(A,B)$ in $H$ and
$E(A',B')$ in $H'$, at least one of the graphs $G[A]$, $G[B]$, $G[A']$
and $G[B']$ is a triangle. The rest of the proof follows the lines of
the proof of Lemma~\ref{lm-split-5-same}.
\end{proof}

We can now prove the lemmas in the C series.

\begin{proof}[Proof of Lemma C.$a.b$]
Let $G$ be a cyclically 5-edge-connected graph, $e=v_1v_2$ be an edge of
$G$, and $H$ be a $b$-expansion of $G$ with $n$ vertices. Our aim is
to prove that for some $\beta$ depending only on $a$ and $b$, $H$ has
at least $(a+3)n/24-\beta$ perfect matchings avoiding $e$. If $a=0$,
consider the graph $H'$ obtained from $H$ by contracting the
Klee-graphs corresponding to $v_1$ and $v_2$ into two single
vertices. This graph is 3-edge-connected and $e$ is not contained in a
cyclic 3-edge-cut. Moreover, $H'$ has at least $n-2b+2$ vertices, so
by Lemma~\ref{lm-3conn}, $H'$ has at least $n/8-(b-1)/4$ perfect
matchings avoiding $e$, and all of them extend to perfect matchings of
$H$ avoiding $e$. The result follows if $\beta \ge (b-1)/4$.

Assume that $a\ge1$, and let $\beta_E$ be the constant from
Lemma~E.$(a-1).b$.  Further, let $v_3$ and $v'_3$ be the neighbors of
$v_2$ different from $v_1$, let $v_4$ and $v_5$ be the neighbors of
$v_3$ different $v_2$, and let $v_4'$ and $v_5'$ be the neighbors of
$v'_3$ different $v_2$.  Consider the graphs $G_1$, $G_2$, $G_3$ and
$G_4$ obtained from $G$ by splitting off the paths $v_1v_2v_3v_4$,
$v_1v_2v_3v_5$, $v_1v_2v'_3v'_4$ and $v_1v_2v'_3v_5'$ and after
possible drop of at most four vertices (replacing two triangles with
vertices) to obtain a cyclically $4$-edge-connected graph. Let $e$ also denote the new edges $v_1v_4$ in $G_1$,
$v_1v_5$ in $G_2$, $v_1v_4'$ in $G_3$ and $v_1v_5'$ in $G_4$. By
Lemmas~\ref{lm-split-5-same} and~\ref{lm-split-5-diff}, at least three
of the graphs $G_i$, say $G_1$, $G_2$, and $G_3$, are cyclically
$4$-edge-connected, and by Lemma~\ref{lm-splitoff} the graph $G_4$ is
3-edge-connected and $e$ is not contained in a cyclic 3-edge-cut of
$G_4$.

\begin{figure}[htbp]
\begin{center}
\includegraphics[scale=0.4]{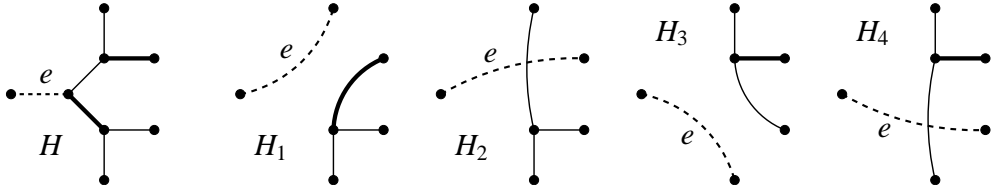}
\caption{A perfect matching of $H$ avoiding $e$ and the corresponding
perfect matchings of $H_1$, $H_3$ and $H_4$ avoiding
$e$. \label{fig:split}}
\end{center}
\end{figure}

For every $1 \le i \le 4$, let $H_i$ be the $b$-expansion of $G_i$
corresponding to $H$ (expand the vertices present both in $G$ and
$G_i$, i.e., all the vertices but at most $8$ vertices removed for
$G_i$, $i=1,2,3$ and 2 vertices removed from $G_4$).  In particular,
$H_1$, $H_2$ and $H_3$ have at least $n-8b$ vertices and $H_4$ has at
least $n-2b$ vertices.  By Lemma E.$(a-1).b$, each of the graphs
$H_1$, $H_2$ and $H_3$ contains at least
$$\tfrac{a+2}{24}(n-8b)-\beta_E$$ perfect matchings avoiding $e$ and
the graph $H_4$ contains at least $(n-2b)/8$ such perfect
matchings by Lemma~\ref{lm-3conn}. Observe that every perfect matching
of $H$ avoiding $e$ corresponds to a perfect matching in
at most three if the graphs $H_1$, $H_2$, $H_3$, and $H_4$ (see Figure~\ref{fig:split} for
an example, where the perfect matchings are represented by thick
edges). We obtain that $H$ contains at least
$$\frac13\cdot\left(3\cdot\left(\tfrac{a+2}{24}\,(n-8b)-\beta_E
\right)+ \tfrac{n-2b}{8}\right)= \tfrac{a+3}{24}\,n-
\beta_E-\tfrac{b}{12}(4a+9)$$ perfect matchings avoiding
$e$. The assertion of the lemma now follows by taking 
$\beta = \max\{ \beta_E+b(4a+9)/12,(b-1)/4\}$.
\end{proof}

\section{Cutting cyclically 4-edge-connected graphs}\label{sec:cutc4c}

Consider a $4$-edge-cut $E(A,B)=\{e_1,\ldots,e_4\}$ of a cubic graph
$G$, and let $v_i$ be the end-vertex of $e_i$ in $A$. Let $
\{i,j,k,\ell\}$ be a permutation of $\{1,2,3,4\}$. The graph
$G^A_{ij}$ is the cubic graph obtained from $G[A]$ by adding two edges
$e_{ij}$ and $e_{k\ell}$ betwen $v_i$ and $v_j$ and between $v_k$ and
$v_{\ell}$. The graph $G^A_{(ij)}$ is the cubic graph obtained from
$G[A]$ by adding one vertex $v_{ij}$ adjacent to $v_i$ and $v_j$, one
vertex $v_{k\ell}$ adjacent to $v_k$ and $v_{\ell}$, and by joining
$v_{ij}$ and $v_{k\ell}$ by an edge denoted by $e^A_{(ij)}$. The edge
between $v_i$ and $v_{ij}$ is denoted by $e^A_i$. We sometimes write
$G_{ij}$, $G_{(ij)}$ and $e_{(ij)}$ instead of $G^A_{ij}$,
$G^A_{(ij)}$ and $e^A_{(ij)}$ when the side of the cut is clear from
the context. The constructions of these two types of graphs are
depicted in Figure~\ref{fig:4cut}.

\begin{figure}[htbp]
\begin{center}
\includegraphics[scale=0.8]{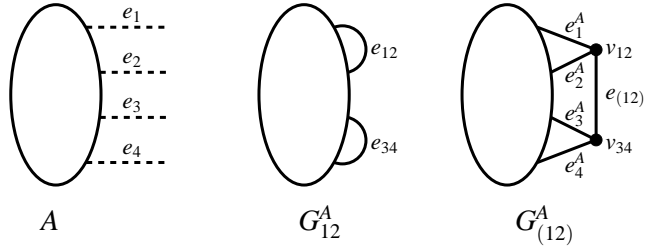}
\caption{The graphs $G^A_{12}$ and $G^A_{(12)}$. \label{fig:4cut}}
\end{center}
\end{figure}

\begin{lemma}
\label{lm-split-4A}
Let $G$ be a cyclically $4$-edge-connected graph and $E(A,B)$ a cyclic
$4$-edge-cut in $G$. All three graphs $G^A_{(12)}$, $G^A_{(13)}$ and
$G^A_{(14)}$ are $3$-edge-connected with any of the edges $e_i^A$ not
being contained in a cyclic $3$-edge-cut. If $G[A]$ is not a cycle of
length of four, then at least two of these graphs are cyclically
$4$-edge-connected.
\end{lemma}

\begin{proof}
Recall that according to Observation~\ref{obs:cyc}, any 2-edge-cut in
a cubic graph is cyclic. Hence, 3-edge-connected and cyclically
3-edge-connected is the same for cubic graphs.

First we prove that all three graphs $G^A_{(12)}$, $G^A_{(13)}$ and
$G^A_{(14)}$ are $3$-edge-connected. Assume $G^A_{(12)}$ has a
(cyclic) 2-edge-cut $E(C,D)$. If both $v_{12}$ and $v_{34}$ are in
$D$, then $E(C,D^\prime\cup B)$ is a 2-edge-cut in $G$ (where
$D^\prime = D\setminus\{v_{12},v_{34}\}$), which is a contradiction
with $G$ being cyclically 4-edge-connected, see Figure \ref{fig:c4e1},
left. Therefore, by symmetry we can assume $v_{12}\in C$ and
$v_{34}\in D$. Let $C^\prime = C\setminus \{v_{12} \}$ and $D^\prime =
D\setminus \{v_{34}\}$, see Figure \ref{fig:c4e1}. Then
$E(C^\prime,D^\prime\cup B)$ is a cyclic 3-edge-cut in $G$ unless
$C^\prime$ contains no cycle, which can happen only if it consists of
a single vertex. Similarly we conclude that $D^\prime$ consists of a single
vertex. But then $A$ has no cycle, which is a contradiction.

\begin{figure}[htbp]
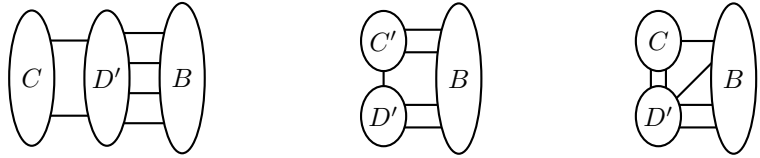

\begin{center}
\includegraphics{c4e.1}
\hfil
\includegraphics{c4e.2}
\hfil
\includegraphics{c4e.3}
\caption{Smaller cyclic edge-cuts of $G$ in the proof of Lemma
\ref{lm-split-4A}. \label{fig:c4e1}}
\end{center}
\end{figure}

Next, we prove that none of the edges of $e_i^A$ is contained in a
cyclic 3-edge cut in $G^A_{(12)}$, $G^A_{(13)}$ or $G^A_{(14)}$. For
the sake of contradiction, assume $G^A_{(12)}$ has a cyclic 3-edge-cut
$E(C,D)$ containing $e_1$. By symmetry, suppose $v_1\in C$ and
$v_{12}\in D$. We claim that $v_2$ and $v_{34}$ belong to $D$: if not,
moving $v_{12}$ from $D$ to $C$ yields a 2-edge-cut in
$G^A_{(12)}$. Let $D^\prime = D \setminus\{v_{12},v_{34}\}$, see
Figure \ref{fig:c4e1}, right. Then $E(C,D^\prime\cup B)$ is a cyclic
3-edge-cut in $G$, a contradiction again.

Finally, assume that $G^A_{(12)}$ and $G^A_{(13)}$ are not cyclically
4-edge-connected. Let $E(C,D)$ and $E(C^\prime,D^\prime)$ be cyclic
3-edge-cuts in $G^A_{(12)}$ and $G^A_{(13)}$, respectively. Just as
above, it is easy to see that $v_{12}$ and $v_{34}$ ($v_{13}$ and
$v_{24}$) do not belong to the same part of the cut $(C,D)$ (the cut
($C^\prime,D^\prime)$ respectively). Let $v_{12}\in C$, $v_{34}\in D$,
$v_{13}\in C^\prime$, $v_{24}\in D^\prime$. Using the same arguments
as in the previous paragraph we conclude that $v_1\in C\cap C^\prime$,
$v_2\in C\cap D^\prime$, $v_3\in D\cap C^\prime$, $v_4\in D\cap
D^\prime$.

For each $(X,Y)\in \{C,D\} \times \{C^\prime,D^\prime\}$ the number of
edges leaving $X\cap Y$ in $G$ is at least 3 (with equality if and
only if $X\cap Y$ consists of a single vertex). As from each of the
four sets $X\cap Y$ there is a single edge going to $B$, the number of
edges among the four sets within $G[A]$ is at least $\frac12(4\cdot 2)
= 4$.  On the other hand, since $E(C,D)$ and $E(C^\prime,D^\prime)$
are 3-edge-cuts and the edges $e_{(12)}$ and $e_{(13)}$ are contained
in the cuts, the number of edges among the four sets within $G[A]$ is
at most 4. Hence, all four sets $X\cup Y$ consist of a single vertex
$v_i$ and there are no edges between $C\cap C^\prime$ and $D\cap
D^\prime$, and between $C\cup D^\prime$ and $C^\prime \cup D$. Since
there can be no parallel edges in $G$, for the other four pairs of $X$
and $Y$ there is precisely one edge between the corresponding vertices
in $X$ and $Y$. It is easy to see that in this case $G[A]$ is a cycle
of length four.
\end{proof}

\begin{lemma}
\label{lm-split-4B}
Let $G$ be a cyclically $4$-edge-connected graph and $E(A,B)$ a cyclic
$4$-edge-cut in $G$.  If $G[A]$ is neither a cycle of length of four
nor the 6-vertex graph depicted in Figure~\ref{fig:exept}, then at
least one of the following holds:
\begin{itemize}
\item all three graphs $G^A_{(12)}$, $G^A_{(13)}$ and $G^A_{(14)}$ are
cyclically $4$-edge-connected,
\item for some $2 \le i \ne j \le 4$, the graphs $G^A_{(1i)}$,
$G^A_{(1j)}$ and $G^A_{1i}$ are cyclically $4$-edge-connected.
\end{itemize}
\end{lemma}

\begin{figure}[htbp]
\begin{center}
\includegraphics[scale=0.6]{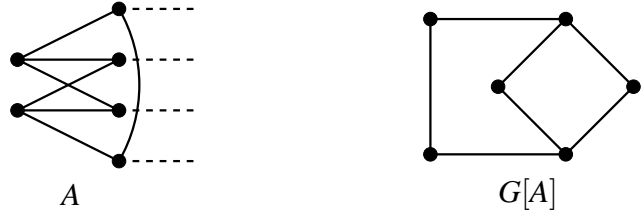}
\caption{The exceptional graph of
Lemma~\ref{lm-split-4B}. \label{fig:exept}}
\end{center}
\end{figure}

\begin{proof}
We assume that $G[A]$ is not a cycle of length four. For the sake of
contradiction, suppose that $G^A_{(12)}$ and $G^A_{(13)}$ are
cyclically 4-edge-connected, but $G^A_{12}$, $G^A_{13}$, and
$G^A_{(14)}$ are not. Let $E(C_2,D_2)$, $E(C_3,D_3)$, and $E(C_4,D_4)$
be cyclic 2- or 3-edge-cuts in $G^A_{12}$, $G^A_{13}$, and
$G^A_{(14)}$, respectively. Note that $G^A_{12}$ and $G^A_{13}$ can
contain 2-edge-cuts.

Consider the vertices $v_1$, $v_2$, $v_3$, and $v_4$. If at least
three of them are in $D_2$, then $E(C_2,D_2\cup B)$ is a cyclic 2- or
3-edge-cut in $G$. If $v_1$ and $v_3$ are in $C_2$ and $v_2$ and $v_4$
in $D_2$, then $E(C_2,D_2\cup B)$ is a cyclic 2- or 3-edge-cut in
$G$ again. Therefore, by symmetry we can assume $v_1,v_2\in C_2$ and
$v_3,v_4\in D_2$. Analogously, we can assume $v_1,v_3\in C_3$ and
$v_2,v_4 \in D_3$.  Using the same arguments as in the proof of Lemma
\ref{lm-split-4A} we conclude that $v_1,v_4,v_{14}\in C_4$ and
$v_2,v_3,v_{23}\in D_4$.  Hence, the sets $X_1=C_2\cap C_3\cap C_4$,
$X_2=C_2\cap D_3\cap D_4$, $X_3=D_2\cap C_3\cap D_4$, and $X_4=D_2\cap
D_3\cap C_4$ are non-empty, since they contain $v_1$, $v_2$, $v_3$,
and $v_4$, respectively.

Let $X_5=D_2\cap D_3\cap D_4$, $X_6=D_2\cap C_3\cap C_4$, $X_7=C_2\cap
D_3\cap C_4$, $X_8=C_2\cap C_3\cap D_4$.  Let $d(X)$ be the number of
edges leaving a vertex set $X$ in $G$. We have $d(X_i)\ge 3$ for each
$i$ such that $X_i$ is non-empty, in particular for $i=1,2,3,4$ (with
equality if and only if $X_i$ consists of a single vertex).  As
from each of the four sets $X_1$, $X_2$, $X_3$, $X_4$ there is a
single edge going to $B$, the number of edges among the eight sets
$X_i$ ($1\le i \le 8)$ within $G[A]$ is at least $\frac12(4\cdot 2+ k
\cdot 3) = 4+\frac32 k$, where $k$ is the number of non-empty sets
$X_i$ for $i=5,6,7,8$.

On the other hand, the number of edges among the eight sets is at most
8, since there are three 3-edge-cuts, and the edge $e_{(14)}$ is in
$E(C_4,D_4)$.  Therefore, $k\le 2$.

If $k=0$, the number of edges among the four sets $X_i$, $i=1,2,3,4$,
is at least 4. On the other hand, each edge is counted in precisely
two cuts, thus the number of edges is exactly 4 and the four sets are
singletons. In this case $G[A]$ is a cycle of length four, a
contradiction.

Assume that $k=1$ and fix $i\in\{1,2,3,4\}$ such that $X_{4+i}$ is
non-empty.  The number of edges among the five non-empty sets is at
least $6>4+\frac32$. On the other hand, each edge from $X_{i}$ (there
are at least 2 such edges) is counted in at least two cuts, thus, the
number of edges is at most $8-2=6$. Therefore, the number of edges is
precisely 6 and four of the five sets are singletons.  Moreover,
precisely two edges are contained in two edge-cuts and four edges are
contained in precisely one edge-cut. The four edges can only join
$X_{i+4}$ to some of the sets $X_1,X_2,X_3,X_4$ except for
$X_i$. Hence, $X_{i+4}$ contains at least two vertices, thus,
$X_1,X_2,X_3,X_4$ are singletons. Since there are no edges between
$X_i$ and $X_{i+4}$, there are at least two edges between $X_{i+4}$
and some $X_j$, $j\ne i$. But then there are at most three edges
leaving $X_{i+4}\cup X_j$ (which contains at least three vertices) in
$G$, a contradiction with $G$ being cyclically 4-edge-connected.


Assume now that $k=2$ and let $X_{i+4}$ and $X_{j+4}$, $1\le i < j \le
4$ be non-empty. The number of edges among the six non-empty sets is
at least $7=4+\frac32\cdot 2$ and at most 8. If the number of edges is
8, each of them is contained in one edge-cut only.  Then the edges
leaving $X_i$ (there are at least two of them) can only end in
$X_{j+4}$, and the edges leaving $X_j$ can only end in $X_{i+4}$.
Since the edges between $X_i$ and $X_{j+4}$ and between $X_j$ and
$X_{i+4}$ belong to the same cut, this cut contains at least four
edges, a contradiction.  Therefore, the number of edges is 7; all the
six sets are singletons and precisely one edge belongs to two
cuts. Since there can be at most one edge between any two sets, the
edge contained in two cuts is the edge from $v_i\in X_i$ to $v_j\in
X_j$. The remaining six edges are the three edges from $v_{i+4}\in
X_{i+4}$ to all $v_k$, $k\in\{1,2,3,4\}\setminus\{i\}$ and the three
edges from $v_{j+4}\in X_{j+4}$ to all $v_k$,
$k\in\{1,2,3,4\}\setminus\{j\}$. The graph $G[A]$ is thus isomorphic
to the exceptional graph depicted in Figure \ref{fig:exept}.
\end{proof}

Cyclic 4-edge-cuts containing a given edge $e$ in a cyclically 4-edge-connected graph turn out to be linearly ordered:

\begin{lemma}
\label{lm-ordered}
Let $G$ be a cyclically $4$-edge-connected graph, and $e$ an edge contained in a cyclic $4$-edge-cut of $G$. There exist $A_1\subseteq A_2 \subseteq \cdots \subseteq A_k$
and $B_i=V(G)\setminus A_i$, $i=1,\ldots,k$, such that every cyclic
$4$-edge-cut of $G$ containing $e$ is of the form $E(A_i,B_i)$.
\end{lemma}

\begin{proof}
Consider two cyclic 4-edge-cuts $E(A,B)$ and $E(A',B')$, such that the
end-vertices of $e$ lie in $A \cap A'$ and $B \cap B'$
respectively. Observe that $A, B,A',B'$ all induce 2-edge-connected
graphs. In order to establish the lemma, it is
enough to show that $A \cap B'=\varnothing$ or $A' \cap B =
\varnothing$. If this is not the case, then for every $X \in \{A \cap
A', B \cap B' \}$ and $Y \in \{A \cap B', B \cap A' \}$ there are at
least two edges between $X$ and $Y$. This implies that $E(A,B)$ and
$E(A',B')$ both contain at least four edges distinct from $e$, a
contradiction.
\end{proof}

\section{Proof of D-series of lemmas}\label{section-D}

The idea to prove the D-series of the lemmas will be to split the
graphs along cyclic 4-edge-cuts, play with the pieces to be sure that
they are cubic with decent connectivity, apply induction on the
pieces, and combine the perfect matchings in the different
parts. However, we will see that combining perfect matchings will be quite
difficult whenever a 2-edge-cut appears in one of the sides of a
cyclic 4-edge-cut. Most of the results in this section
(Lemmas~\ref{lm-twisted-struc} to~\ref{lm-twisted-num-bis}) will allow
us to overcome this difficulty.

\begin{lemma}
\label{lm-triple}
If $G$ is a cyclically $4$-edge-connected cubic bipartite graph with
at least 8 vertices, then every edge is contained in at least 3
perfect matchings of $G$.
\end{lemma}

\begin{proof}
Let $e=uv$ be an edge of $G$. Observe that the graph $H=G\setminus
\{u,v\}$ has minimum degree two, since otherwise $G$ would contain a
cyclic edge-cut of size two or three. Since $G$ is cubic and
bridgeless, it has a perfect matching containing $e$, and $H$ has a
perfect matching $M$. Our aim is to find two different (but not
necessarily disjoint) alternating cycles in $H$ with respect to
$M$. This will prove that $H$ has at least three perfect matchings,
which will imply that $G$ has at least three perfect matchings
containing $e$.

Let $f$ be any edge contained in $M$. Start marching from $f$ in any
direction, alternating the edges in $M$ and the edges not in $M$ until
you hit the path you marched on. Since $G$ is bipartite, this yields
an alternating cycle. If the cycle does not contain $f$, start
marching from $f$ in the other direction and obtain a different
alternating cycle. If the cycle contains $f$, consider an edge not
contained in the alternating cycle (such an edge exists since $H$ has
at least six vertices) and start marching on it until you hit a vertex
visited before; this yields another alternating cycle.
\end{proof}

Let $G$ be a cyclically $4$-edge-connected graph.
If $E(A,B)$ is a cyclic $4$-edge-cut,
we say that $B$ is {\em solid} if $G[B]$ does not have a $2$-edge-cut
with at least two vertices on each of its sides,
in particular, $G[B]$ must have at least eight vertices.

For a graph containing only vertices of degree two and three, the
vertices of degree two are called \emph{corners}. If a graph is
comprised of a single edge, its two end-vertices are also called
corners. We call {\em twisted net} a graph being either a 4-cycle, or
the graph inductively obtained from a twisted net $G$ and a twisted
net (or a single edge) $H$ by adding edges $uv$ and $u'v'$ to the
disjoint union of $G$ and $H$, where $u,u'$ and $v,v'$ are corners of
$G$ and $H$, respectively.  If $H$ is a single edge, this operation is
called an \emph{incrementation}; it is the same as adding a path of
length three between two corners of $G$. If $H$ is a twisted net, the
operation is called a \emph{multiplication}. Observe that every
twisted net has exactly four corners, and that the special graph on
six vertices depicted in Figure~\ref{fig:exept} is a twisted net. The
following lemma will be useful in the proof of lemmas in Series D:

\begin{lemma}
\label{lm-twisted-struc}
Let $G$ be a cyclically $4$-edge-connected graph with a distinguished
edge $e$ that is not contained in any cyclic $4$-edge-cut. If for
every cyclic $4$-edge-cut $E(A,B)$ with $e \in G[A]$, $B$ is not
solid, then for each such cut $G[B]$ is a twisted net.
\end{lemma}

\begin{proof}
Proceed by induction on the number of vertices in $G[B]$.
If the number of vertices of $B$ is at most six, the claim clearly holds.
If $B$ has more than six vertices, it can be split into parts
as $B$ is not solid. If they both contain a cycle, the claim
follows by induction. Otherwise, one of them contains a cycle and
the other is just an edge; the claim again follows by induction.
\end{proof}

Our aim is now to prove that twisted nets have an exponential number
of perfect matchings (Lemma~\ref{lm-twisted-num}), and an exponential
number of matchings covering all the vertices except two corners
(Lemma~\ref{lm-twisted-num-bis}). In order to prove this second result
we need to consider the special case of bipartite twisted nets, and
prove stronger results about them (Lemma~\ref{lm-twisted-num-bip}).

\begin{lemma}
\label{lm-twisted-num}
If $G$ is a twisted net with $n$ vertices, then $G$ has at least
$2^{n/18+2/3}$ perfect matchings.
\end{lemma}

\begin{proof}
We proceed by induction on $n$. First assume that $G$ was obtained
from a single 4-cycle by a sequence of $k \le 6$ incrementations. If
$k \le 1$, then $n\le 6$ and $G$ has at least $2 \ge 2^{n/18+2/3}$
perfect matchings. If $2 \le k \le 6$ it can be checked that $G$ has
at least 3 perfect matchings. Since $n\le 16$, we have $3 \ge
2^{n/18+2/3}$ and the claim holds.  Assume now that there exist two
twisted nets $H_1$ and $H_2$ on $n_1$ and $n_2$ vertices respectively,
so that $G$ was obtained from at most six incrementations of the
multiplication of $H_1$ and $H_2$. In this case $n_1+n_2 \ge n -12$
and by the induction, $G$ has at least $2^{n_1/18+2/3} \cdot
2^{n_2/18+2/3} \ge 2^{n/18+2/3} $ perfect matchings.

So we can now assume that $G$ was obtained from a twisted net $H_0$ by
a sequence of seven incrementations, say $H_1, \ldots , H_7=G$. For a
twisted net $H_i$ with corners $v_i^1,\ldots,v_i^4$, and for any
$X\subseteq \{1,\ldots, 4\}$ define the quantities $m^{H_i}_X$ as the
number of perfect matchings of $H_i \setminus \{v_i^j, j \in
X\}$. Assume that $H_1$ is obtained (without loss of generality) by
adding the path $v_0^1v_1^2v_1^1v_0^2$ to $H_0$, and set $v_1^3=v_0^3$
and $v_1^4=v_0^4$. We observe that
$m^{H_1}_{\varnothing}=m^{H_0}_{\varnothing}+m^{H_0}_{12}$ and
$m^{H_1}_{12}=m^{H_0}_{\varnothing}$. Moreover, for every pair
$\{i,j\} \subset \{1,2,3,4\}$ distinct from $\{1,2\}$, we have that
$m^{H_1}_{ij} \ge m^{H_0}_{ij}$. Therefore, $m^{H_7}_{\varnothing} \ge
2\, m^{H_0}_{\varnothing}$. As a consequence, $G$ has at least $2
\cdot 2^{(n-14)/18+2/3} \ge 2^{n/18+2/3} $ perfect matchings, which
concludes the proof of Lemma~\ref{lm-twisted-num}.
\end{proof}

\begin{lemma}
\label{lm-twisted-num-bip}
Let $G$ be a bipartite twisted net with $n$ vertices. Then $G$ has a
pair of corners in each color class, say $u_1,u_2$ and $v_1,v_2$, and
the graphs $G \setminus \{u_1,u_2,v_1,v_2\}$ and $G \setminus
\{u_i,v_j\}$ have a perfect matching for any $i,j\in \{1,
2\}$. Moreover, for some $i,j\in \{1, 2\}$, the graph $G \setminus
\{u_i,v_j\}$ has at least $2^{n/18-2/9}$ perfect matchings.
\end{lemma}

\begin{proof}
The fact that each color class contains two corners of $G$, as well as
the existence of perfect matchings of $G \setminus
\{u_1,u_2,v_1,v_2\}$, and $G \setminus \{u_i,v_j\}$ for any $i,j\in
\{1, 2\}$, easily follow by induction on $n$ (we consider that the
empty graph contains a perfect matching): if $G$ was obtained from a
twisted net $H$ by an incrementation, a matching avoiding all four
corners of $G$ is the same as a matching avoiding two corners of
different colors in $H$ (which is assumed to exist by the
induction). A matching avoiding two corners of different colors in $G$
is either a perfect matching of $H$ or a matching avoiding two corners
of different colors in $H$. So we can assume that $G$ was obtained
from $H_1$ and $H_2$ by a multiplication. In this case a matching
avoiding all four corners of $G$ can be obtained by combining
matchings avoiding all four corners in $H_1$ and $H_2$. Let $u_1,v_1$
be the corners of $G$ lying in $H_1$ and $u_2,v_2$ be the corners of
$G$ lying in $H_2$. First assume that $u_1,v_1$ are in one color class
of $G$, and $u_2,v_2$ are in the other one. In this case, matchings of
$G$ avoiding two corners of different colors are obtained by combining
matchings of $H_1$ and $H_2$ avoiding two corners of different
colors. Otherwise, since $G$ is bipartite, it means without loss of
generality that $u_1,u_2$ are in one color class, and $v_1,v_2$ are in
the other color class. A perfect matching of $G \setminus \{u_1,v_1\}$
is then obtained by combining a perfect matching of $H_1 \setminus
\{u_1,v_1\}$ and a perfect matching of $H_2$. Let $w_1$ be the corner
of $H_1$ of the same color as $v_1$, and let $w_2$ be the corner of
$H_2$ with the same color as $u_2$. A perfect matching of $G \setminus
\{u_1,v_2\}$ is obtained by combining a perfect matching of $H_1
\setminus \{v_1,w_1\}$ and a perfect matching of $H_2 \setminus
\{v_2,w_2\}$. All other matchings of $G$ avoiding two corners of
different colors are obtained in one of these two ways.

Consider now the graph $H$ obtained from $G$ by adding two adjacent
vertices $u,v$ and by joining $u$ to $v_1,v_2$ and $v$ to
$u_1,u_2$. This graph is cubic, bridgeless, and bipartite, so by
Theorem~\ref{thm-bip} it has at least $(4/3)^{(n+2)/2}$ perfect
matchings avoiding the edge $uv$. As a consequence, two corners of $G$
in different color classes, say $u_1,v_1$ are such that $G\setminus
\{u_1,v_1\}$ has at least $\frac14 (4/3)^{(n+2)/2} \ge 2^{n/6-5/3}$
perfect matchings (we use that $2^{1/3} \le 4/3$). If $n=4$, $G$ has
at least $1=2^{4/18-2/9}$ matching avoiding two corners. If $6 \le n
\le 12$, it can be checked that $G$ has at least $2 \ge 2^{n/18-2/9}$
matchings avoiding two corners. If $n \ge 14$, $2^{n/6-5/3} \ge
2^{n/18-1/9} \ge 2^{n/18-2/9}$, which concludes the proof.
\end{proof}

\begin{lemma}
\label{lm-twisted-num-non-bip}
Suppose $G$ is a non-bipartite twisted net with $n$ vertices and
corners $v_1,\ldots,v_4$. If for every $1 \le i < j \le 4$, we denote
by $m^G_{ij}$ the number of perfect matchings of $G \setminus
\{v_i,v_j\}$, then $\prod_{1 \le i < j \le 4} m^G_{ij} \ge
2^{n/18+4/9}$. In particular, all values $m^G_{ij}$ are at least one.
\end{lemma}

\begin{proof}
We prove the statement by induction on $n$. There is only one non
bipartite twisted net of order at most six (it is the special graph of
Figure~\ref{fig:exept}). In this graph, one value $m^G_{ij}$ is two
and the others are one. Hence, the product of the $m^G_{ij}$ is at
least $2 \ge 2^{6/18+4/9}$.  Assume now that $G$ was obtained from $H$
by adding a path $v'_1v_2v_1v'_2$ between $v'_1$ and $v'_2$. By
Lemma~\ref{lm-twisted-num}, $G\setminus \{v_1,v_2\}$ has at least
$2^{(n-2)/18+2/3} \ge 2^{n/18+4/9}$ perfect matchings. So we only have to
make sure that all the other values $m^G_{ij}$ are at least one. If
the graph $H$ is not bipartite, then by the induction, for any pair
$\{ x,y \}$ of corners of $G$ distinct from $\{v_1,v_2 \}$, the graph
$G \setminus \{x,y\}$ also has a perfect matching. If $H$ is
bipartite, then $v'_1$ and $v'_2$ must lie in the same color class. By
Lemma~\ref{lm-twisted-num-bip}, $H$ has a matching covering all the
vertices except the four corners, and matchings covering all the
vertices except any two corners belonging to different color
classes. All these matchings extend to perfect matchings of $G
\setminus \{x,y\}$ for any pair of corners $\{ x,y \}$ distinct from
$\{v_1,v_2 \}$.

So we can assume that $G$ was obtained from two twisted nets $H_1$ and
$H_2$ of order $n_1,n_2$ by a multiplication. Let $v_1,v_3$ be the two
corners of $G$ lying in $H_1$, and let $v_2,v_4$ be the two corners of
$G$ lying in $H_2$. If none of $H_1,H_2$ is bipartite, then by
induction it is easy to check that $m^G_{ij}\ge 1$ for all $1 \le i <
j \le 4$. Moreover, since $H_1 \setminus \{v_1,v_3\}$ has a perfect
matching, $G\setminus \{v_1,v_3\}$ has at least $2^{n_2/18+2/3}$
perfect matchings by Lemma~\ref{lm-twisted-num}. Similarly,
$G\setminus \{v_2,v_4\}$ has at least $2^{n_1/18+2/3}$ perfect
matchings. As a consequence, $$\prod_{1 \le i < j \le 4} m^G_{ij} \ge
2^{n_2/18+2/3} \cdot 2^{n_1/18+2/3} \ge 2^{n/18+4/3} \ge
2^{n/18+4/9}.$$

Assume now that one of $H_1,H_2$, say $H_1$, is bipartite, while the
other is not bipartite. Denote by $u_1,u_3$ the corners of $H_1$
distinct from $v_1,v_3$, in such way that the graphs $H_1 \setminus
\{u_1,v_1\}$ and $H_1 \setminus \{u_3,v_3\}$ both have a perfect
matching (this is possible by Lemma~\ref{lm-twisted-num-bip}). Also
denote by $u_2$ and $u_4$ the corners of $H_2$ adjacent to $u_1$ and
$u_3$ in $G$, respectively. Observe that the perfect matchings of $H_2
\setminus \{v_2,v_4\}$ combine with perfect matchings of $H_1$ to give
perfect matchings of $G \setminus \{v_2,v_4\}$, and that perfect
matchings of $H_2 \setminus \{u_2,u_4\}$ combine with perfect
matchings of $H_1\setminus \{u_1,v_1,u_3,v_3\}$ (their existence is
guaranteed by Lemma~\ref{lm-twisted-num-bip}) to give perfect
matchings of $G \setminus \{v_1,v_3\}$. Also observe that for any $i
\in \{1,3\}$ and $j \in \{2,4\}$, a perfect matching of $G \setminus
\{v_i,v_j\}$ can be obtained by combining perfect matchings of $H_1
\setminus \{v_i,u_i\}$ and $H_2 \setminus \{u_{i+1},v_j\}$. As a
consequence, $$\prod_{1
\le i < j \le 4} m^G_{ij} \ge 2^{n_1/18+2/3} \cdot \prod_{1
\le i < j \le 4} m^{H_2}_{ij} \ge 2^{n_1/18+2/3} \cdot 2^{n_2/18+4/9}\ge
2^{n/18+4/9}.$$

Assume now that $H_1,H_2$ are both bipartite. Since $G$ is not
bipartite, without loss of generality it means that $v_1,v_3$ have
different colors in $H_1$ whereas $v_2,v_4$ have the same color in
$H_2$. Using that $H_2$ has a perfect matching and a matching covering
all the vertices except the four corners, and that both $H_1$ and
$H_2$ have matchings covering all the vertices except any two corners
in different color classes gives that for any pair $\{u,v\} \subset
\{v_1,v_2,v_3,v_4\}$, $G\setminus \{u,v\}$ has a perfect
matching. Hence, all values $m^G_{ij}$ are at least one. Again, we
denote by $u_1,u_3$ the corners of $H_1$ distinct from $v_1,v_3$, and
by $u_2$ and $u_4$ the corners of $H_2$ adjacent to $u_1$ and $u_3$ in
$G$, respectively. By Lemma~\ref{lm-twisted-num-bip}, without loss of
generality one of $H_1 \setminus \{v_1,v_3\}$, $H_1 \setminus
\{u_1,u_3\}$, and $H_1 \setminus \{v_1,u_1\}$ has at least
$2^{n_1/18-2/9}$ perfect matchings. If $H_1 \setminus \{v_1,v_3\}$ has
at least $2^{n_1/18-2/9}$ perfect matchings, then by combining them
with perfect matchings of $H_2$ we obtain at least $2^{n_1/18-2/9}
\cdot 2^{n_2/18+2/3} \ge 2^{n/18+4/9}$ perfect matchings of
$G\setminus \{v_1,v_3\}$. Assume that this is not the case, then we
still obtain at least $2^{n_2/18+2/3}$ such perfect matchings since
$v_1,v_3$ have different colors in $H_1$. If $H_1 \setminus
\{u_1,u_3\}$ has at least $2^{n_1/18-2/9}$ perfect matchings, they
combine with perfect matchings of $H_2 \setminus \{u_2,v_2,u_4,v_4\}$
to give at least $2^{n_1/18-2/9}$ perfect matchings of $G\setminus
\{v_2,v_4\}$. If $H_1 \setminus \{v_1,u_1\}$ has at least
$2^{n_1/18-2/9}$ perfect matchings, they combine with perfect
matchings of $H_2 \setminus \{u_2,v_2\}$ to give at least
$2^{n_1/18-2/9}$ perfect matchings of $G\setminus \{v_1,v_2\}$. In any
case, $$\prod_{1 \le i < j \le 4} m^G_{ij} \ge 2^{n_1/18-2/9} \cdot
2^{n_2/18+2/3} \ge 2^{n/18+4/9}.$$
\end{proof}

Lemmas~\ref{lm-twisted-num-bip} and~\ref{lm-twisted-num-non-bip} have
the following immediate consequence:

\begin{lemma}
\label{lm-twisted-num-bis}
If $G$ is a twisted net with $n$ vertices, then there exist two
corners $u,v$ of $G$ such that $G\setminus \{u,v\}$ has at least
$2^{n/108-1/27}$ perfect matchings.
\end{lemma}

We now use these results to prove the D series of the lemmas.

\begin{proof}[Proof of Lemma D.$a.b$]
  Let $G$ be a cyclically 4-edge-connected graph, $e$ an edge of $G$
  not contained in a cyclic 4-edge-cut, and $H$ a $b$-expansion of $G$
  with $n$ vertices. Our aim is to prove that for some $\beta$
  depending only on $a$ and $b$, $H$ has at least $(a+3)n/24-\beta$
  perfect matchings avoiding $e$. If $a=0$, then the lemma follows
  from Lemma~\ref{lm-3conn} with $\beta=(b-1)/4$ (see the proof of the
  C series).  Assume now that $a\ge 1$. Let $\beta_B$ be the constant
  from Lemma B.$a$, $\beta_C$ the constant from Lemma C.$a.b$ and
  $\beta_E$ the constant from Lemma E.$(a-1).b$, and set $\beta$ to be
  the maximum of the numbers $2\beta_B+22$, $(a+3)b/6+\beta_B$,
  $\beta_C$, $(a+3)b/2+3\beta_E+30$, $21(a+3)b\cdot\ln 42
  (a+3)^2b+2+\beta_E$, and $\tfrac{a+3}{24}\,\kappa(a,b)$ (with
  $\kappa(a,b)$ depending only on $a$ and $b$, to be defined later in
  the proof). The proof proceeds by induction on the number of
  vertices of $G$.

If $G$ is cyclically $5$-edge-connected, the claim follows from Lemma
C.$a.b$. Assume that $G$ has a cyclic $4$-edge-cut $E(A,B)$ such that
$e$ is contained in $G[A]$ and at least one of the following holds:
\begin{enumerate}
\item
[(1)]
$G[A]$ is a cycle of length four,
\item
[(2)]
$G[A]$ is the six-vertex exceptional graph of Figure~\ref{fig:exept}, 
\item
[(3)] 
 $B$ is a twisted net of size at least $k$, where $k$ is the
  smallest integer such that $2^{k/108-1/27}\ge (a+3)n/24$, or
\item
[(4)]
$B$ is solid.
\end{enumerate}
Let $E(A^*,B^*)$ be the edge-cut of $H$ so that $H[A^*]$ and $H[B^*]$
are the expansions of $G[A]$ and $G[B]$; let $n_A$ and $n_B$ be the
numbers of vertices of $H[A^*]$ and $H[B^*]$.  Let $e_1$, $e_2$, $e_3$
and $e_4$ be the edges of $E(A,B)$, let $v_1$, $v_2$, $v_3$ and $v_4$
be their end-vertices in $A$, and let $v^H_1$, $v^H_2$, $v^H_3$ and
$v^H_4$ be their endvertices in $H[A^*]$.  For
$X\subseteq\{1,2,3,4\}$, let $g^A_X$ ($h^A_X$) denote the number of
matchings of $G[A]$ ($H[A^*]$) avoiding $e$ and covering all the
vertices of $G[A]$ ($H[A^*]$) except $v_i$ ($v^H_i$), $i\in
X$. Similarly, $g^B_X$ ($h^B_X$) is used. For each of these types of
matchings in $H[A^*]$ and $H[B^*]$ fix two matchings to be canonical
(if they exist, if not fix at least one if possible) and for
$X=\varnothing$, fix three matchings to be canonical (if they exist,
if not, fix as many as possible). Let $H^A_{ij}$ and $H^A_{(ij)}$
($H^B_{ij}$ and $H^B_{(ij)}$) be the expansions of $G^A_{ij}$ and
$G^A_{(ij)}$ ($G^B_{ij}$ and $G^B_{(ij)}$), respectively, for
$\{i,j\}\subset\{1,2,3,4\}$.

First assume that $G[A]$ is a cycle of length four. Without loss of
generality, the edge $e$ joins the end-vertices of $v_1$ and
$v_4$. Let $H'$ be the graph obtained from $H$ by contracting the
expansions of the vertices of $A$ into 4 single vertices. This graph
has at least $n-4b$ vertices, and each perfect matching of $H^B_{12}$
can be combined with a matching of $G[A]$ avoiding $e$ to give a
perfect matching of $H'$ avoiding $e$. Hence, $H$ has at
least $$\tfrac{a+3}{24}\,n_B-\beta_B \ge \tfrac{a+3}{24}\,n-
\tfrac{a+3}{6}\,b-\beta_B$$ perfect matchings avoiding $e$.

The case that $G[A]$ is the six-vertex exceptional graph of
Figure~\ref{fig:exept} will be addressed later in the proof. 

Consider now the third case. If $g^A_\varnothing\ne 0$, then by Lemma
\ref{lm-twisted-num} the graph $G[B]$ has at least
$2^{n^G_B/18+2/3}\ge 2^{n^G_B/108-1/27} \ge (a+3)n/24$ perfect
matchings, where $n^G_B$ is the number of vertices of $G[B]$; all such
perfect matchings extend
to perfect matchings of $H$. 

Assume $g^A_\varnothing= 0$. By Lemma \ref{lm-twisted-num-bis},
$g^B_{ij} \ge 2^{n^G_B/108-1/27} \ge (a+3)n/24$ for some
$\{i,j\}\subset\{1,2,3,4\}$.  By Lemma~\ref{lm-split-4A}, we may
assume that the graphs $G^A_{(12)}$ and $G^A_{(13)}$ are cyclically
$4$-edge-connected. Since there are no perfect matchings in
$G^A_{(12)}$ containing the edge $e^A_{12}$ and avoiding $e$, by Lemma
\ref{lm-special} the graph $G[A]\setminus e$ is bipartite and $e$
joins two vertices of the same color class. Then by Lemma
\ref{lm-special}, $G^A_{(13)}$ has a perfect matching containing
$e^A_1$ and avoiding $e^A_4$. Such perfect matchings must contain
$e^A_2$ and avoid $e$, thus $g^A_{12}\ne 0$.  Similarly, we obtain
that all the quantities $g^A_X$ with $|X|=2$ are non-zero. Therefore,
we can extend the matchings of $G[B]$ avoiding the vertices $v^B_i$
and $v^B_j$ to perfect matchings
of $H$.\\

We now analyse case (4). Assuming that $A$ contains at least 6
vertices and $B$ is solid, we will estimate the numbers of perfect
matchings of $H$ canonical in one part and non-canonical in the other.
We start with matchings canonical in $H[A^*]$ and non-canonical in
$H[B^*]$ and show that there are at least $(a+3)n_B/24-\beta/2$ such
perfect matchings in $H$.

We first assume that $G[A]\setminus e$ is not a bipartite graph such
that $e$ joins two vertices of the same color. By
Lemma~\ref{lm-split-4B}, we can assume that one of the following two
cases apply: all the graphs $G^A_{(12)}$, $G^A_{(13)}$ and
$G^A_{(14)}$ are cyclically $4$-edge-connected, or all the graphs
$G^A_{(12)}$, $G^A_{(13)}$ and $G^A_{12}$ are cyclically
$4$-edge-connected.

Let us first deal with the case that the graphs $G^A_{(1i)}$ with
$i=2,3,4$ are cyclically $4$-edge-connected.  Since neither of these
graphs can be of the form described in Lemma~\ref{lm-special}, there
exists a perfect matching of $G^A_{(1i)}$ containing $e_{(1i)}$ and
avoiding $e$ and so $g^A_{\varnothing}\ge 1$. In addition, for any
distinct $i,j,k \in \{1,2,3,4\}$, $g^A_{ij}+g^A_{ik}\ge 1$, since
there exists a perfect matching of $G^A_{(jk)}$ containing $e^A_i$ and
avoiding $e$. Hence, by symmetry, we can assume that all the
quantities $g^A_{13}$, $g^A_{14}$, $g^A_{23}$ and $g^A_{24}$ are
non-zero.  Now, since $H^B_{(12)}$ is a cubic bridgeless graph, by
Lemma B.$a$ it has at least $(a+3)n_B/24-\beta_B$ perfect matchings,
which all extend to $H[A^*]$.  At most $11$ of these matchings are
canonical in $H[B^*]$ and thus the number of perfect matchings avoiding $e$
canonical in $H[A^*]$ and non-canonical in $H[B^*]$ is at least
$(a+3)n_B/24-\beta_B-11$.

We now consider the case when the graphs $G^A_{(12)}$, $G^A_{(13)}$
and $G^A_{12}$ are cyclically $4$-edge-connected. As in the previous
case, $g^A_{\varnothing}$ is non-zero. If $g^A_{14}$ or $g^A_{23}$ is
zero, then we conclude that all the quantities $g^A_{12}$, $g^A_{13}$,
$g^A_{24}$ and $g^A_{34}$ are non-zero and proceed as in the previous
case. Hence, we can assume that both $g^A_{14}$ and $g^A_{23}$ are
non-zero. If $g^A_{1234}$ is also non-zero, we consider the graph
$H^B_{14}$ and argue that each of its perfect matchings can be
extended to $H[A^*]$ and obtain the bound. Finally, if $g^A_{1234}$ is
zero, then by considering matchings in $G^A_{12}$ containing $e_{12}$
and matchings containing $e_{34}$ we obtain that both $g^A_{12}$ and
$g^A_{34}$ are non-zero. In this case, all the perfect matchings of
the graph $H^B_{(13)}$ extend to $H[A^*]$ and the result follows.\\

We can now assume that the graph $G[A]\setminus e$ is bipartite (with
color classes $U,V$) and $e$ joins two vertices in the same color
class, say $U$. By degree counting argument, we obtain that it can be
assumed without loss of generality that $v_1 \in U$ and $v_2,v_3,v_4
\in V$, or $v_1,v_2,v_3,v_4 \in U$.

In the first case, we can assume by Lemma~\ref{lm-split-4A} that the
graphs $G^A_{(12)}$ and $G^A_{(13)}$ are cyclically
$4$-edge-connected. By Lemma~\ref{lm-double}, $G^A_{(12)}$ is double
covered, so it has two perfect matchings containing the edge
$e_{(12)}$. Since these two perfect matchings avoid the edge $e^A_2$,
they also avoid $e$ by Lemma~\ref{lm-special} and so
$g^A_{\varnothing}\ge 2$. By Lemma~\ref{lm-special}, $G^A_{(12)}$ has
a perfect matching containing $e^A_1$ and avoiding $e^A_i$ for
$i=3,4$. Since such perfect matchings avoid $e^A_2$, they also avoid
$e$. Hence, we obtain that $g^A_{13}$ and $g^A_{14}$ are non-zero. A
similar argument for the graph $G^A_{(13)}$ yields that also
$g^A_{12}$ is non-zero. Consider now perfect matchings avoiding the
edge $e^B_i$ in $H^B_{(1i)}$, for $i=2,3,4$. By
Lemma~\ref{lm-split-4A}, two of the graphs $G^B_{(1i)}$ are cyclically
4-edge-connected; by Lemma~E.$(a-1).b$ there are at least
$(a+2)n_B/24-\beta_E$ perfect matchings avoiding $e^B_i$ in
$H^B_{(1i)}$. The third graph $G^B_{(1i)}$ is cyclically
3-edge-connected and $e^B_i$ is not contained in a cyclic
3-edge-cut. Its expansion $H^B_{(1i)}$ is cyclically 3-edge-connected,
too, and the only cyclic 3-edge-cut containing $e^B_i$ is the cut
separating the expansion of $v^B_i$ from the rest of the graph. Let
$H'$ be the graph obtained from $H^B_{(1i)}$ by contraction of the
Klee-graph corresponding to $v^B_i$ in $H^B_{(1i)}$ to a single
vertex. The graph $H'$ has at least $n_B-b$ vertices; it is cyclically
3-edge-connected and $e^B_i$ is not contained in a cyclic
3-edge-cut. Hence, by Lemma~\ref{lm-3conn}, the number of perfect
matchings of $H^B_{(1i)}$ avoiding $e^B_i$ is at least
$(n_B-b)/8$. Altogether, we get
$$2\,h^B_{12}+2\,h^B_{13}+2\,h^B_{14}+3\,h^B_{\varnothing} \ge 2\cdot
\tfrac{a+2}{24}\,n_B + \tfrac1{8}\,n_B -\tfrac18\,b-2\beta_E.$$ As a consequence,
non-canonical matchings of $H[B^*]$ can be combined with canonical
matchings of $H[A^*]$ avoiding $e$ to give at least $$ h^B_{12}+
h^B_{13}+ h^B_{14}+2\,h^B_{\varnothing} -12\ge \tfrac{a+3}{24}\,n_B
-\tfrac1{16}\,b-\beta_E -12$$ perfect matchings of $H$ avoiding $e$.

We now assume that $v_1,v_2,v_3,v_4 \in U$. Again, it can be
assumed that the graphs $G^A_{(12)}$ and $G^A_{(13)}$ are cyclically
$4$-edge-connected. An application of Lemma~\ref{lm-special} similar
to the one in the previous paragraph yields that all the quantities
$g^A_X$ with $|X|=2$ are non-zero. Since $B$ is solid, all the graphs
$G^B_{(ij)}$ with $\{i,j\}\subseteq\{1,2,3,4\}$ are cyclically
$4$-edge-connected.  Hence, each $H^B_{(ij)}$ contains at least
$(a+2)n_B/24-\beta_E$ perfect matchings avoiding the edge
$e^B_{(ij)}$. As a consequence,
$$2\,h^B_{12}+2\,h^B_{13}+2\,h^B_{14}+ 2\,h^B_{23}+2\,
h^B_{24}+2\,h^B_{34}\ge 3 \cdot \tfrac{a+2}{24}\,n_B-3\beta_E.$$
Subtracting 12 matchings canonical in $H[B^*]$, we obtain that the number
of perfect matchings avoiding $e$ that are canonical in $H[A^*]$ and
non-canonical in $H[B^*]$ is at least
$$ \tfrac32 \cdot \tfrac{a+2}{24}\,n_B-\tfrac32\, \beta_E-12 \ge
\tfrac{a+3}{24}\,n_B-\tfrac32 \beta_E-12 .$$ This
concludes the counting of perfect matchings of $H$ avoiding $e$ that
are canonical in $H[A^*]$ and non-canonical in $H[B^*]$.\\

Observe that the bound just above also holds if $G[A]$ is the
exceptional six-vertex graph of Figure~\ref{fig:exept}. The edge $e$
cannot be a part of the 4-cycle (otherwise the first case would
apply), nor be adjacent to it (otherwise $e$ is contained in a cyclic
4-edge-cut in $G$). Hence, $G[A]\setminus e$ is bipartite,
$v_1,v_2,v_3,v_4$ have the same color and in particular $e$ connects
two vertices of the same color. In this case, since $n\le n_B+6b$, $H$
has at least $$ \tfrac{a+3}{24}\,n-\tfrac{a+3}4\,b-\tfrac32\,
\beta_E-12$$ perfect matchings avoiding $e$. So from now on we can
assume that $G[A]$ is neither a 4-cycle nor the exceptional six-vertex
graph of
Figure~\ref{fig:exept}.\\

We will now count the perfect matchings of $H$ that are
non-canonical in $H[A^*]$ and canonical in $H[B^*]$. Our aim is to show
that there are at least $(a+3)n_A/24-\beta/2$ such matchings.  
%

Consider the graphs $G^A_{(12)}$, $G^A_{(13)}$ and $G^A_{(14)}$.  Two
of these graphs are cyclically $4$-edge-connected by
Lemma~\ref{lm-split-4A}; the remaining one is $3$-edge-connected. We
claim it has no cyclic $3$-edge-cut containing $e$. Assume
$G^A_{(12)}$ has a cyclic $3$-edge-cut $E(C,D)$ containing $e$. It is
clear that the new edge $e^A_{(12)}$ belongs to the cut; let $f$ be
the third edge of the cut. Then $\{e,f,e_1,e_2 \}$ and
$\{e,f,e_3,e_4\}$ are $4$-edge-cuts in $G$ containing $e$. Since $G$
has no cyclic $4$-edge-cuts containing $e$, both $C\cap A$ and $D\cap
A$ consist of a pair of adjacent vertices. Then $G[A]$ is a cycle of
length 4, which was excluded above.

Lemmas E.$(a-1).b$ and~\ref{lm-3conn} now imply that
$$2h^A_{12}+2h^A_{13}+2h^A_{14}+2h^A_{23}+2h^A_{24}+2h^A_{34}+3h^A_{\varnothing}
\ge 2\cdot\tfrac{a+2}{24}\,n_A-2\beta_E+\tfrac1{8}\,(n_A-2b).$$ By the
choice of $B$ as solid, all the graphs $G^B_{(12)}$,
$G^B_{(13)}$ and $G^B_{(14)}$ are cyclically
$4$-edge-connected. In particular, if none
of them is the exceptional graph described in Lemma~\ref{lm-special},
then all the quantities $g^B_X$ with $|X|=2$ are non-zero and
$g^B_{\varnothing}\ge 2$ (here we use that cyclically 4-edge-connected
graphs are double covered). The bound now follows by dividing the
previous inequality by two and subtracting the at most $18$ canonical
matchings.

Otherwise, exactly two of the three graphs are of the form described
in Lemma~\ref{lm-special}, and $G[B]$ is bipartite. By symmetry, we
can assume that $v_1$ and $v_2$ lie in one color class and $v_3$ and
$v_4$ in the other.  Considering the graphs $G^B_{(13)}$ and
$G^B_{(14)}$, we observe that each of the quantities $g^B_{13}$,
$g^B_{14}$, $g^B_{23}$ and $g^B_{24}$ is at least two as the graphs
$G^B_{(13)}$ and $G^B_{(14)}$ are double covered by
Lemma~\ref{lm-double}.  In addition, Lemma~\ref{lm-triple} applied to
the bipartite graph $G^B_{(12)}$ yields that $g^B_{\varnothing}$ is at
least three. Finally, observe that the graph $G^B_{12}$ satisfies the
conditions of Lemma~\ref{lm-special}. Hence, any perfect matching of
$G^B_{12}$ containing $e_{12}$ also contains $e_{34}$, which implies
that $g^B_{1234}$ is non-zero and the number of matchings
non-canonical in $H[A^*]$ and canonical in $H[B^*]$ is at least
$$2h^A_{13}+2h^A_{14}+2h^A_{23}+2h^A_{24}+3h^A_{\varnothing}+
h^A_{1234}-27.$$  Replace now $B$ with the cycle of length four
$v_1v_3v_2v_4$ and observe that the resulting graph is cyclically
$4$-edge-connected. By Lemma E.$(a-1).b$, its expansion has
$$h^A_{13}+h^A_{14}+h^A_{23}+h^A_{24}+2h^A_{\varnothing}+h^A_{1234} \ge
\tfrac{a+2}{24}\,(n_A+4)-\beta_E$$ perfect matchings avoiding
$e$. Observe also that the graph $G^A_{(12)}$ is 3-edge-connected and
no cyclic 3-edge-cut contains $e$. Its expansion (except for the end-vertices of $e$) has
at least $(n_A-2b)/8$ perfect matchings avoiding $e$, thus,
$$h^A_{13}+h^A_{14}+h^A_{23}+h^A_{24}+h^A_{\varnothing} \ge
\tfrac1{8}\,(n_A-2b).$$ Summing the two
previous inequalities, we obtain that the number of perfect matchings
avoiding $e$ that are non-canonical in $H[A^*]$ and canonical in $H[B^*]$
is at least $$\tfrac{a+2}{24}\,(n_A+4)+\tfrac1{8}\,n_A-\tfrac14\,b-\beta_E -27\ge
\tfrac{a+3}{24}\,n_A-\tfrac12 \beta.$$

The bound on the number of matchings now follows from the estimates
on the perfect matchings canonical in one of the graphs $H[A^*]$ and $H[B^*]$
and non-canonical in the other. This finishes the first part of the
proof of Lemma D.$a$.$b$.\\

Based on the analysis above, we may now assume that $|A|\ge 8$ and if
$E(A,B)$ is a cyclic $4$-edge-cut of $G$ and $e$ is contained in $A$,
then $G[B]$ is a twisted net of size less than $k$, where $k$ is the
smallest integer such that $2^{k/108-1/27}\ge (a+3)n/24$ (see
Lemma~\ref{lm-twisted-struc}).  In particular, consider such a cyclic
$4$-edge-cut $E(A,B)$ with $B$ inclusion-wise maximal.  Assume that
$G[B]$ is a non-bipartite twisted net. Then by Lemma
\ref{lm-twisted-num-non-bip} we have $g^B_X\ge 1$ for any
$X\subset\{1,2,3,4\}$ with $|X|=2$.  Moreover, by Lemma
\ref{lm-twisted-num}, $g^B_\varnothing \ge 2$. Then there are at least
$$
h^A_{12}+h^A_{13}+h^A_{14}+h^A_{23}+h^A_{24}+h^A_{34}+2h^A_\varnothing
$$
perfect matchings avoiding $e$ in $H$.  Consider the graphs
$G^A_{(12)}$, $G^A_{(13)}$ and $G^A_{(14)}$.  Two of these graphs are
cyclically $4$-edge-connected by Lemma~\ref{lm-split-4A}; the
remaining one is $3$-edge-connected and it has no cyclic $3$-edge-cut
containing $e$.  Since $|B|<k$, their expansions have at least $n-kb$
vertices.  Hence, Lemmas E.$(a-1).b$ and~\ref{lm-3conn} imply that
$$
\gathered
h^A_{12}+h^A_{13}+h^A_{14}+h^A_{23}+h^A_{24}+h^A_{34}+2h^A_\varnothing \ge \\
\ge \tfrac12\,\left(2\cdot\tfrac{a+2}{24}\,(n-kb)-2\beta_E+\tfrac1{8}\,(n-kb-2b)\right)\ge \\
\ge \tfrac{a+3}{24}\,n +\tfrac1{48}\,n- \tfrac{b}{8}\,(a+3)(k+1).
\endgathered
$$

Assume that $G[B]$ is a bipartite twisted net. Let $e_1,\ldots,e_4$ be
the edges of the cut ordered in such a way that matchings including
$e_i$ and $e_{i+1}$, $i=1,2,3,4$, indices modulo four, extend to
$G[B]$ by Lemma~\ref{lm-twisted-num-bip}. Moreover, $g^B_{1234}\ge 1$
and $g^B_\varnothing\ge 2$.  Then there are at least
$$
h^A_{12}+h^A_{14}+h^A_{23}+h^A_{34}+h^A_{1234}+2h^A_\varnothing 
$$
perfect matchings avoiding $e$ in $H$. Let $m_{12}$, $m_{14}$, and
$m_{(13)}$ be the number of perfect matchings avoiding $e$ in the
graphs $H^A_{12}$, $H^A_{14}$, and $H^A_{(13)}$, respectively. Then
$$
h^A_{12}+h^A_{14}+h^A_{23}+h^A_{34}+h^A_{1234}+2h^A_\varnothing \ge
\tfrac12\,(m_{12}+m_{14}+m_{(13)}).
$$

In the rest of this section, we show that we can assume that at least
one of the following two cases applies:
\begin{enumerate}
\item [(1)] $G^A_{(13)}$ and one of the graphs $G^A_{12}$ and
  $G^A_{14}$ (say $G^A_{12}$) are $(2k+3)$-almost cyclically
  $4$-edge-connected, and $G^A_{14}$ is 3-edge-connected with no
  cyclic $3$-edge-cut containing $e$, or
\item [(2)] $G^A_{(13)}$ and one of the graphs $G^A_{12}$ and
  $G^A_{14}$ (say $G^A_{12}$) are $(2k+3)$-almost cyclically
  $4$-edge-connected, and the vertex set of $G^A_{14}$ can be
  partitioned into three parts $X$, $Y$ and $Z$ such that $E(X,Y\cup
  Z)$ and $E(X\cup Y,Z)$ are cyclic $3$-edge-cuts containing $e$,
  $G^A_{14}[Y]$ is a twisted net with $|Y| <k$, and both the graphs
  $G^A_{14}/(X\cup Y)$ and $G^A_{14}/(Y\cup Z)$ are $3$-edge-connected
  with no cyclic $3$-edge-cut containing $e$.
\end{enumerate}

Observe that the expansions of $G^A_{12}$, $G^A_{14}$ and $G^A_{(13)}$
have at least $n-kb$ vertices. In the first case, we apply
Lemma~E.$(a-1).b$ to the first two graphs and Lemma~\ref{lm-3conn} to
the remaining one, obtaining that $\tfrac12\,(m_{12}+m_{14}+m_{(13)})$
is at least
$$
\gathered
\tfrac12 \cdot \left( 2 \cdot \tfrac{a+2}{24}\,(n-(2k+3)b-kb)-2\beta_E -\tfrac{n-kb-2b}8\right)\ge \\ \ge
\tfrac{a+3}{24}\,n +\tfrac1{48}\,n- \tfrac{b}{8}\,(a+3)(k+1)-\beta_E.\endgathered$$

In the second case, we again apply Lemma~E.$(a-1).b$ to the first two
graphs.  Let $e_1$ and $e_2$ ($e_3$ and $e_4$) be the edges joining
$Y$ to $X$ ($Z$, respectively) in the third graph, say $G^A_{14}$. Let
$v_1,v_2,v_3,v_4$ be end-vertices of $e_1,e_2,e_3,e_4$ in $Y$.
According to Lemmas~\ref{lm-twisted-num-bip} and
\ref{lm-twisted-num-non-bip}, without loss of generality we may assume
that $G^A_{14}[Y]\setminus \{v_1,v_3 \}$ and $G^A_{14}[Y]\setminus
\{v_2,v_4 \}$ both have perfect matchings. Let $h^X_i$ be the number
of perfect matchings containing $e_i$, $i=1,2$, in the graph obtained
by $G^A_{14}/(Y\cup Z)$ by expanding as in $H$ all the vertices except
for the end-vertex of $e$. Observe that such graph does not contain a
cyclic 3-edge-cut containing $e$. Let $h^Z_i$, $i=3,4$ be defined
analogously. Let $n_X$ and $n_Z$ be the numbers of vertices in the
(full) expansions of $G^A_{14}/(Y\cup Z)$ and $G^A_{14}/(X\cup
Y)$. Since $|Y|\le k$ and $|B|\le k$, the number of perfect matchings
of $G^A_{14}$ avoiding $e$ is at least
$$
\gathered
h^X_1\cdot h^Z_3 + h^X_2 \cdot h^Z_4 \ge 
h^X_1+h^X_2 + h^Z_3+h^Z_4-2 \ge \\
\ge \tfrac18\,(n_X-b) + \tfrac18\,(n_Z-b)\ge \tfrac18\,(n-2kb-2b)-2.
\endgathered
$$
In this case, 
$\tfrac12\,(m_{12}+m_{14}+m_{(13)})$ is at least 
$$\gathered\tfrac12 \cdot \left( 2 \cdot \tfrac{a+2}{24}\,(n-(2k+3)b-kb)-2\beta_E -\tfrac{n-2kb-2b}8-2\right)\ge\\\ge
\tfrac{a+3}{24}\,n +\tfrac1{48}\,n- \tfrac{b}{8}\,(a+3)(k+1)-2-\beta_E.\endgathered$$

Observe that $2^{(k-1)/108-1/27} < \tfrac{a+3}{24}\,n$. Then using
$2^{168}>e^{108}$ and the fact that $e^x\ge 1+x$ for all
$x\in\mathbb{R}$ we get
$$
\aligned
\tfrac{1}{48}\,n 
&=\tfrac{1}{2(a+3)}\cdot \tfrac{a+3}{24}\,n \ge \\
&\ge \tfrac{1}{2(a+3)}\cdot 2^{(k-1)/108-1/27} 
=\tfrac{1}{2(a+3)}\cdot 2^{(k-5)/108} > \\
&>\tfrac{1}{2(a+3)}\cdot e^{(k-5)/168} 
=21(a+3)b\cdot e^{(k-5)/168 - \ln 42 (a+3)^2b} \ge \\
&\ge 21(a+3)b\cdot\left(1+ \tfrac{k-5}{168} - \ln 42 (a+3)^2b\right) > \\
&>\tfrac{b}{8}\,(a+3)(k+1) - 21(a+3)b\cdot\ln 42 (a+3)^2b. 
\endaligned
$$
The claim follows by the choice of $\beta$.\\


We now prove that (1) or (2) holds. Assume that one of the graphs
$G^A_{(12)}$, $G^A_{(13)}$, and $G^A_{(14)}$ is not $4$-almost
cyclically $4$-edge-connected, or one of the graphs $G^A_{12}$,
$G^A_{13}$ and $G^A_{14}$ is not $3$-edge-connected. Then, without
loss of generality $G[A]$ contains a 2-edge-cut $E(C,D)$ so that $e$,
$v_1$, and $v_2$ are in $C$, and $v_3$ and $v_4$ are in $D$. By
maximality of $B$, the 4-edge-cut $E(C, D\cup B)$ of $G$ is not cyclic
and $C$ consists of a single edge $e=v_1v_2$. On the other hand $E(D,
C\cup B)$ is a cyclic 4-edge-cut of $G$ (since otherwise $A$ would be
a 4-cycle), so $G[D]$ is a twisted net of size less than $k$. As a
consequence $G$ has at most $2k+2 \le 216\,
\log_2(\tfrac{a+3}{24}\,n)+12$ vertices. Since it has at least $n/b$
vertices, we obtain that $n$ is upper-bounded by
a constant $\kappa(a,b)$ depending only on $a$ and $b$ (which we do not compute
here, since the computation is very similar to the previous
one). Taking $\beta$ to be at least $\tfrac{a+3}{24}\, \kappa(a,b)$
yields the desired bound on the number of perfect matchings of $H$
avoiding $e$.
Therefore, we can assume in the following that the
graphs $G^A_{(12)}$, $G^A_{(13)}$, and $G^A_{(14)}$ are $4$-almost
cyclically $4$-edge-connected, and the graphs $G^A_{12}$, $G^A_{13}$,
and $G^A_{14}$ are $3$-edge-connected.

We now show that at least one of the graphs $G_{12}$ and $G_{14}$ is
$(2k+3)$-almost cyclically $4$-edge-connected. Assume that this is not
the case. Since the cyclic 3-edge-cuts of $G_{1i}$ correspond to
cyclic 4-edge-cuts in $G_{(1i)}$ containing $e_{1i}$,
Lemma~\ref{lm-ordered} implies that they are linearly
ordered. Therefore, $G_{12}$ contains a cyclic 3-edge-cut $E(C,D)$ and
$G_{14}$ contains a cyclic 3-edge-cut $E(C',D')$ such that all the
sets $C,D,C',D'$ have size at least $k+2\ge 4$. Without loss of
generality, we can assume that $v_1 \in C\cap C'$, $v_2 \in C'\cap D$,
$v_3 \in D\cap D'$, and $v_4 \in C\cap D'$.
\begin{figure}[htbp]
\begin{center}
\includegraphics[scale=0.5]{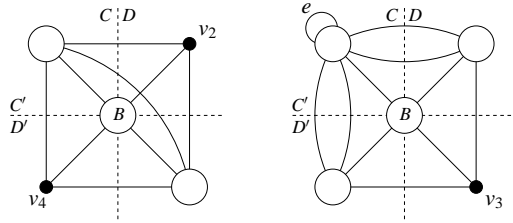}
\caption{In the case where none of $G_{12}$ and $G_{14}$ is
$(2k+3)$-almost cyclically $4$-edge-connected.\label{fig:dser2}}
\end{center}
\end{figure}

Assume there is an edge between $C\cap C'$ and $D\cap D'$. Beside this
edge, there are at most four more edges among the four sets $X\cap Y$,
$X\in\{C,C'\}$, $Y\in\{D,D'\}$. On the other hand, there are at least
two edges leaving $C\cap D'$ and at least two edges leaving $D\cap
C'$. Hence, there are precisely two edges leaving both $C\cap D'$ and
$D\cap C'$, $C\cap D'$ and $D\cap C'$ are $\{v_4\}$ and $\{v_2\}$
respectively, and $C\cap C'$ and $D\cap D'$ have size at least $k+1
\ge 3$ (see Figure~\ref{fig:dser2}, left). Hence, the edge-cuts
leaving $C\cap C'$ and $D\cap D'$ are cyclic 4-edge-cuts by
Observation~\ref{obs:cyc}. Since $e$ is not in a cyclic 4-edge-cut,
$e$ must lie in $C\cap C'$ or $D\cap D'$. In both cases, this
contradicts the maximality of $B$.

Consequently, we can assume without loss of generality that there are
no edges between $C\cap C'$ and $D\cap D'$, and between $C\cap D'$ and
$D\cap C'$.  Hence, all six edges of the cuts are within $C$, $C'$,
$D$, and $D'$. Without loss of generality, we may assume that there is
at most one $(C',D')$-edge in $D$ and at most one $(C,D)$-edge in
$D'$.  It means $D \cap D'$ contains a single vertex $\{v_3\}$, $C\cap
D'$ and $D\cap C'$ have size at least $k+1\ge 3$ (see
Figure~\ref{fig:dser2}, right). By maximality of $B$, $e$ is neither
in $C\cap D'$ nor in $D \cap C'$. The edges leaving $C \cap D'$ form a
cyclic 4-edge-cut, so by our assumption, $G[C \cap D']$ is a twisted
net of size at most $k$, a contradiction. This proves that one of
$G_{12}$ and $G_{14}$, say $G_{14}$, is $(2k+3)$-almost cyclically
$4$-edge-connected.

Assume now that $G_{12}$ has a cyclic $3$-edge-cut containing $e$.
Observe that cyclic $3$-edge-cuts of $G_{12}$ containing $e$
one-to-one correspond to such cyclic $4$-edge-cuts of $G_{(12)}$ and
apply Lemma~\ref{lm-ordered} to $G_{(12)}$. Set $X=A_1$, $Z=B_k$ and
$Y$ to be the remaining vertices. Clearly, $Y$ must be a twisted net
of size less than $k$. Observe that each of $G/(X\cup Y)$ and
$G/(Y\cup Z)$ is cyclically $3$-edge-connected. By minimality of $X$
and $Z$, $e$ is not contained in a cyclic 3-edge-cut in any of these
two graphs, as claimed.
\end{proof}

\section{Proof of E-series of lemmas}\label{section-E}

This section is mainly devoted to counting perfect matchings avoiding
an edge contained in a cyclic 4-edge-cut. A {\em ladder} of height $k$
is a $2 \times k$ grid. The two edges of a ladder having both
end-vertices of degree two are called the \emph{ends} of the ladder.

\begin{lemma}
\label{lm-ladder}
Let $G$ be a cyclically $4$-edge-connected graph and $E(A,B)$ a cyclic
$4$-edge-cut of $G$ containing the edges $e_1,\ldots,e_4$ having
end-vertices $v_1,\ldots,v_4$ in $A$. For $1 \le i \ne j \le 4$, let
$g^A_{ij}$ be the number of matchings of $G[A]$ covering all the
vertices of $A$ except for $v_i,v_j$. If one of the three numbers
$g^A_{23}$, $g^A_{24}$, and $g^A_{34}$ is zero, say $g^A_{ij}$, then
either the other two are at least two, or one of them is one, say
$g^A_{ik}$, and the subgraph $G[A]$ is a ladder with ends $v_1v_i$ and
$v_jv_k$.
\end{lemma}

\begin{proof}
Fix $G$ and choose an inclusion-wise minimal set $A$ in $G$ that does
not satisfy the statement of the lemma. By Lemma~\ref{lm-split-4A}, we can assume
that the graphs $G^A_{(12)}$ and $G^A_{(13)}$ are cyclically
$4$-edge-connected. By considering the matchings including the edges
$e_2$ and $e_3$ in these two graphs, we obtain that
$g^A_{23}+g^A_{24}$ and $g^A_{23}+g^A_{34}$ are at least two since
every cyclically $4$-edge-connected graph is double covered by
Lemma~\ref{lm-double}. Hence, if $g^A_{23}=0$ then
$g^A_{24}\ge 2$ and $g^A_{34}\ge 2$. By symmetry we can now assume that $g^A_{24}=0$
and so $g^A_{23} \ge 2$. In order to prove the lemma, we only need to
show that either $g^A_{34} \ge 2$, or $g^A_{34}=1$ and $G[A]$ is a
ladder with ends $v_1v_4$ and $v_2v_3$.

By Lemma~\ref{lm-special}, there exists a proper $2$-coloring of the
vertices of $G[A]$ such that $v_1$ and $v_3$ are in one color class, say $C_1$,
while $v_2$ and $v_4$ are in the other class, say $C_2$. Consequently, the graph
$G_{(13)}$ is bipartite.  By Lemma~\ref{lm-special}, $G_{(13)}$
contains a matching avoiding $e_1$ and containing $e_4$, i.e.,
$g^A_{34}\ge 1$.

Assume that $g^A_{34}=1$. By Lemma~\ref{lm-bridge}, the graph $H$
obtained from $G[A]$ by removing the vertices $v_3$ and $v_4$ has a
bridge contained in the unique perfect matching of $H$. Define the
\emph{deficiency} $d(H)$ of a subcubic graph $H$ to be the sum of the differences
between three and the degrees of the vertices. Since $G^A_{(12)}$ is cyclically 4-edge-connected, the vertices $v_3$ and $v_4$ are not adjacent in $G$, hence, $d(H)$ is six, three in each color class of $H$. Let $V$
and $W$ be such sets that the cut $E(V,W)$ is formed by the bridge $f$ of
$H$. Since the bridge $f$ is contained in the unique perfect matching of $H$, we can
assume that $|V\cap C_1|=|V\cap C_2|+1$ and thus $|W\cap C_1|=|W\cap
C_2|-1$. It means that the subgraphs $G[V]$ and $G[W]$ induced by $V$ and $W$ have odd numbers of vertices, hence, their deficiencies (including the end-vertices of the bridge $f$) are odd.
On the other hand, $d(G[V])$ and $d(G[W])$ cannot be equal to one, otherwise $f$ would be a bridge in $G$.
Since $d(G[V])+d(G[W])=8$, we can assume $d(G[V])=3$ and $d(G[W])=5$. But then the three edges leaving $V$ in $G$ form a cyclic 3-edge-cut, unless $G[V]$ is a single vertex $w$. Then $V\cap C_1=\{w\}$ and $V\cap C_2=\varnothing$; the degree of $w$ in $H$ is one. 

The vertex $w$ is thus either adjacent to $v_3$ or $v_4$, or it is one of the vertices $v_1$ and $v_2$.
Since $v_3$ and $v_4$ are in different color classes, $w$ is not adjacent to both of them. Since $w\in C_1$, $w=v_1$ 
and it is adjacent to $v_4$. 

Let
$A'=A\setminus\{v_1,v_4\}$ and $B'=B\cup\{v_1,v_4\}$. We denote by
$v_1'$ and $v_4'$ the neighbors of $v_1$ and $v_4$ in $A'$. If $G[A]$
is not a cycle of length four, then $E(A',B')$ is a cyclic
4-edge-cut. Observe that $g^A_{24}=0$ and $g^A_{34}=1$ implies
$g^{A'}_{12}=0$ and $g^{A'}_{13}=1$. So, by the minimality of $A$, the
subgraph $G[A']$ is a ladder with ends $v_1'v_4'$ and $v_2v_3$. Hence,
$G[A]$ is a ladder with ends $v_1v_4$ and $v_2v_3$.
\end{proof}

\begin{proof}[Proof of Lemma E.$a.b$]
{The proof proceeds by
induction on the number of vertices in $G$ (in addition, to the
general induction framework).} 

Let $G$ be a cyclically 4-edge-connected graph, $e$ an edge of $G$ and $H$ a $b$-expansion of $G$ with
$n$ vertices. Our aim is to prove that for some $\beta$ depending only
on $a$ and $b$, $H$ has at least $(a+3)n/24-\beta$ perfect matchings
avoiding $e$. If $e$ is not contained in a cyclic 4-edge-cut of $G$,
then this follows from Lemma D.$a.b$, so we can assume in the
remaining of the proof that $e$ is contained in a cyclic 4-edge-cut of
$G$.  

If $a=0$, then the lemma follows from Lemma~\ref{lm-3conn} with
$\beta=b/4$.  Assume that $a>0$ and let $\beta_E$ be the constant from
Lemma E.$(a-1).b$, $\beta_D$ the constant from Lemma D.$a.b$ and
$\beta_B$ the constant from Lemma B.$a$. Let $\gamma$ be the least
element of $\{n \in \mathbb{N} \, | \, n\ge 4\}$ satisfying
$$2^{\gamma/4-2}\ge \tfrac{a+3}{24}\,(\gamma b)+2$$
and $\beta$ be the maximum of the following numbers:
$4 \beta_E-24$, $(a+3)b/4$, $(a+3)\gamma b/12 + \beta_D$,
$(a+3)\gamma b/12+\beta_B$,
$(a+2)\gamma b/8+ 3\beta_E/2$,
$(a+2)(\gamma+1)b/6+2\beta_E$.\\

Let $A_1\subseteq A_2\cdots\subseteq A_k$ and $B_k\subseteq B_{k-1}
\cdots \subseteq B_1$ be as in the statement of
Lemma~\ref{lm-ordered}.  Assume first that there exists $i_0$ such
that neither $G[A_{i_0}]$ nor $G[B_{i_0}]$ is a ladder and they both
contain at least eight vertices each.
 To simplify the presentation, we
will write $A$ instead of $A_{i_0}$ and $B$ instead of $B_{i_0}$. Let
$e_2$, $e_3$, $e_4$, and $e=e_1$ be the edges of the edge-cut
$E(A,B)$. 
As previously, $h^A_X$ denotes the number of matchings of the
expansion $H[A]$ of $G[A]$ covering all the vertices except the
end-vertices of $e^A_i$, $i \in X$. The quantities $h^B_X$ are defined
accordingly for $G[B]$. Finally, let $E(A^*,B^*)$ be the edge-cut of
$H$ so that $H[A^*]$ and $H[B^*]$ are the expansions of $G[A]$ and
$G[B]$, and let $n_A$ and $n_B$ be the number of vertices of $H[A^*]$
and $H[B^*]$.

By Lemma~\ref{lm-split-4B}, without loss of generality at least one of
the following holds:
\begin{itemize}
\item {\it All the graphs $G^A_{(12)}$, $G^A_{(13)}$ and $G^A_{(14)}$ are
      cyclically $4$-edge-connected.}  By inspecting the types of
      perfect matchings avoiding $e=e_1$ in these graphs, we obtain
      that the three quantities $h^A_{23}+h^A_{24}+h^A_{\varnothing}$,
      $h^A_{23}+h^A_{34}+h^A_{\varnothing}$, and
      $h^A_{24}+h^A_{34}+h^A_{\varnothing}$ are at least
      $(a+2)(n_A+2)/24-\beta_E$ by Lemma E$.(a-1).b$.
      
\item {\it All the graphs $G^A_{(12)}$, $G^A_{(13)}$ and $G^A_{12}$
      are cyclically $4$-edge-connected.}  By inspecting the types of
      perfect matchings avoiding $e$ in these graphs, we obtain that
      two quantities $h^A_{23}+h^A_{24}+h^A_{\varnothing}$ and
      $h^A_{23}+h^A_{34}+h^A_{\varnothing}$ are at least
      $(a+2)(n_A+2)/24-\beta_E$, while
      $$h^A_{34}+h^A_{\varnothing} \ge \tfrac{a+2}{24}\,n_A-\beta_E.$$
\end{itemize}
In any case, all the quantities $h^A_{23}+h^A_{24}+h^A_{\varnothing}$,
      $h^A_{23}+h^A_{34}+h^A_{\varnothing}$, and
      $h^A_{24}+h^A_{34}+h^A_{\varnothing}$ are at least
      $(a+2)n_A/24-\beta_E$.

A symmetric argument now yields that all the quantities $h^B_{23}+
      h^B_{24}+h^B_{\varnothing}$,
      $h^B_{23}+h^B_{34}+h^B_{\varnothing}$, and
      $h^B_{24}+h^B_{34}+h^B_{\varnothing}$ are at least
      $(a+2)n_B/24-\beta_E$.

Choose one or two (two if possible) canonical matchings for each of
the four possible types avoiding $e$ ($23$, $24$, $34$, and
$\varnothing$). Since one of the graphs $G^A_{(ij)}$ is cyclically
4-edge-connected, it is double covered by Lemma~\ref{lm-double} and so
$h^A_\varnothing\ge 2$. Similarly, we have $h^B_\varnothing\ge 2$. If
all $h^A_{23}$, $h^A_{24}$ and $h^A_{34}$ are non-zero, then the
number of combinations of a canonical matching in $H[A^*]$ and a
non-canonical matching in $H[B^*]$ is at least
\begin{eqnarray*}
h^B_{23}+h^B_{24}+h^B_{34}+2h^B_{\varnothing}-10 & \ge &
h^B_{23}+h^B_{24}+h^B_{34}+\tfrac32\,h^B_{\varnothing}-10 \\ & \ge &
\tfrac32 \cdot \left(\tfrac{a+2}{24}\,n_B-\beta_E\right) -10 \\ & \ge &
\tfrac{a+3}{24}\,n_B-\tfrac32\,\beta_E -10.
\end{eqnarray*}
If one of the quantities is zero, say $h^A_{34}=0$, then $g^A_{34}=0$ and
Lemma~\ref{lm-ladder} yields $g^A_{23}$ and $g^A_{24}$ (as well as $h^A_{23}$ and $h^A_{24}$) are at least
two since $G[A]$ is not a ladder (recall that we assumed that for the 4-edge-cut
$E(A,B)$ containing $e$, neither $G[A]$ nor $G[B]$ is a ladder). Hence, the
number of combinations of a canonical matching in $H[A^*]$ and a
non-canonical matching in $H[B^*]$ is at least
\begin{eqnarray*}
2\,h^B_{23}+2\,h^B_{24}+2\,h^B_{\varnothing}-12 & \ge &
2\cdot (\tfrac{a+2}{24}\,n_B-\beta_E)-12\\ & \ge &
\tfrac{a+3}{24}\,n_B-2\beta_E-12.
\end{eqnarray*}
Similarly, we estimate combinations of non-canonical matchings in
$H[A^*]$ and canonical matchings of $H[B^*]$ to be at least
$(a+3)n_A/24-2\beta_E-12$. Hence, the expansion of $G$ has at least
$(a+3)n/24-4\beta_E-24$ perfect matchings avoiding $e$.\\

In the rest, we assume that whenever $G[A_i]$ and $G[B_i]$ have at least 8
vertices, at least one of them is a ladder. Assume there is at least
one cut such that both parts have at least 8 vertices.  It is clear that if $G[A_{i_0}]$ is a ladder, then for
all $i\le i_0$ $G[A_i]$ is a ladder too. Analogously, if $G[B_{j_0}]$
is a ladder, then for all $j \ge j_0$ $G[B_j]$ is a ladder too. Let
$i_0$ be the largest $i$ such that $G[A_i]$ is a ladder. Then if $i_o
< k$, $G[A_{i_0+1}]$ is not a ladder, and therefore, $G[B_{i_0+1}]$ is
either a ladder or a graph on at most 6 vertices.

Assume that $G[A_{i_0}]$ is a ladder with at least $\gamma$ vertices
(recall that $\gamma$ was defined as the least integer satisfying
$2^{\gamma/4-2} \ge (a+3)\gamma b/24+2$) and $B_{i_0}$ has at least
eight vertices. We again write $A$ and $B$ instead of $A_{i_0}$ and
$B_{i_0}$. It can be checked that $G[A]$ (as well as $H[A^*]$) has a
matching covering all the vertices except the end-vertices of $e_i$
and $e_j$ for two different pairs $i,j$ in $\{2,3,4\}$ with $i \ne j$,
say $2,3$ and $2,4$. Fix a single canonical matching of $H[A^*]$
avoiding each of these two pairs of vertices, and a single canonical
perfect matching of $H[A^*]$. Fix a single canonical perfect matching
of $H[B^*]$ (such a perfect matching exists since any of the graphs
$G^B_{(ij)}$ is bridgeless, and thus matching-covered). By
Lemma~\ref{lm-split-4B} and the observations in the previous cases,
one of the graphs $G^B_{(12)}$, $G^B_{13}$, or $G^B_{14}$ is
cyclically 4-edge-connected and all perfect matchings of its expansion
avoiding $e$ can be combined with a canonical matching of
$H[A^*]$. Hence, the number of combinations of a canonical matching in
$H[A^*]$ and a non-canonical matching in $H[B^*]$ is at least
$(a+3)n_B/24-\beta-1$ by the induction within this lemma (we
subtracted one to count the canonical matching).

Observe that there
are at least $2^{\lfloor n^G_A/4\rfloor}$ perfect matchings in $G[A]$
containing none of the edges of the cut, where $n^G_A$ is the number of
vertices of $A$ and these at least $$\tfrac{a+3}{24}\,n^G_A \,b+2\ge
\tfrac{a+3}{24}\,n_A+2$$ matchings (the bound follows from the choice
of $\gamma$) can be extended by the canonical matching of
$H[B^*]$. Subtracting one for a possible canonical matching among these,
we obtain that the number of combinations of a non-canonical matching
in $H[A^*]$ and a canonical matching in $H[B^*]$ is at least
$(a+3)n_A/24+1$, which together with the bound on the combinations of
canonical matchings in $H[A^*]$ and non-canonical matchings in $H[B^*]$
yields the desired bound.  

Observe that if $G[A]$ is a ladder with at
least $\gamma$ vertices and $G[B]$ has less than eight vertices, there
are at least $$\tfrac{a+3}{24}\,n_A+2 \ge \tfrac{a+3}{24}\,n -
\tfrac{a+3}{4}\,b + 2 $$ perfect matchings in $H$. This includes the
case when the whole graph is a ladder.\\

\begin{figure}[htbp]
\begin{center}
\includegraphics[scale=0.5]{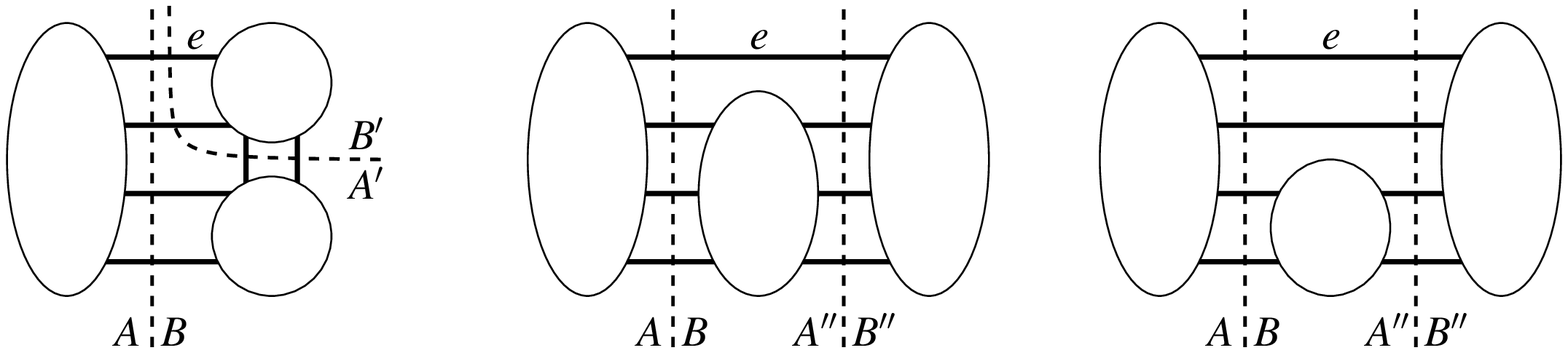}
\caption{Cyclic 4-edge-cuts containing $e$ if $i_0=k$. \label{fig:eser1}}
\end{center}
\end{figure}

For the rest of the proof, we can assume $G[A_{i_0}]$ is a ladder with less than $\gamma$
vertices.  If the number of vertices of $G$ is less than $3\gamma$,
then there is nothing to prove by the choice of $\beta$.

First, assume that $i_0=k$. We again write $A$ and $B$ instead of
$A_{i_0}$ and $B_{i_0}$. Let $E(A,B)=\{e=e_1,e_2,e_3,e_4\}$. Since
$G[A]$ is a ladder, we may assume there is a (canonical) matching of
$H[A^*]$ covering all vertices except the end-vertices of $v_i$ and
$v_j$, $\{i,j\}=\{2,3\}$ or $\{2,4\}$; and a (canonical) perfect
matching of $H[A^*]$.  Consider the graph $G^\prime=G^B_{(12)}$.  If
there is a cyclic 3-edge-cut $E(X,Y)$ in $G^\prime$, then the new edge
$e_{12}$ is in the cut. Assume that the end-vertex of $e$ in $B$ is in
$Y$. Then $E(A^\prime, B^\prime)$ with $A^\prime=A\cup X\setminus
\{v_{34}\}$, $B^\prime =Y\setminus \{v_{12}\}$ is a cyclic 4-edge-cut
in $G$ containing $e$ such that $A^\prime\supsetneq A$, a
contradiction (see Figure~\ref{fig:eser1}, left). Therefore,
$G^\prime$ is cyclically 4-edge-connected.

If $G'$ has a cyclic
4-edge-cut $E(X,Y)$ containing $e$, then the new edge $e_{12}$ is not
in the cut. Again, assume that the end-vertex of $e$ in $B$ is in
$Y$. Then $v_{12},v_{34}\in X$ and again
$E(A^{\prime\prime},B^{\prime\prime})$ with $A^{\prime\prime}=A\cup
X\setminus \{v_{12},v_{34}\}$, $B^{\prime\prime}=Y$ is a cyclic
4-edge-cut in $G$ such that $A^{\prime\prime}\supsetneq A$, a
contradiction (see Figure~\ref{fig:eser1}, center and
right). Therefore, there is no cyclic 4-edge-cut containing $e$ in
$G$.  Hence, by Lemma~D.$a.b$, the expansion of $G^B_{(12)}$ has at
least $$
\tfrac{a+3}{24}\,(n-\gamma b)-\beta_D=\tfrac{a+3}{24}\,n-\tfrac{a+3}{24}\,\gamma b-\beta_D
$$
perfect matchings avoiding $e$.  As each of these matchings can be
extended by a canonical matching of $H[A^*]$ to a perfect matching of
$H$,
the claim now follows by the choice of $\beta$.\\

Next, assume that $i_0<k$. Then $G[A_{i_0+1}]$ is not a ladder, thus
$G[B_{i_0+1}]$ has less than 8 vertices or it is a ladder with less
than $\gamma$ vertices. Let $A=A_{i_0}$, $B=B_{i_0+1}$,
$C=V(G)\setminus(A\cup B)$.  We use the following arguments also in
the case when for all $i$ either $G[A_i]$ or $G[B_i]$ has less than 8
vertices.

The number of edges betwen $A$ and $B$ is one or two: the edge
$e$ is contained in both $(A\cup C,B)$ and $(A,B\cup C)$ and thus it
must be joining a vertex of $A$ and a vertex of $B$. On the other
hand, if they were three or more edges between $A$ and $B$, then there
would be at most two edges between $A\cup B$ and $C$ which is
impossible since $G$ is cyclically $4$-edge-connected.

Assume now that there are exactly two edges between between $A$ and
$B$, and let $e_2$ be the edge distinct from $e$. Let $e_3$ and $e_4$
be the edges between $A$ and $C$ and $e_5$ and $e_6$ the edges between
$B$ and $C$ (see Figure~\ref{fig:eser2}, left). Since $G[A]$ and $G[B]$
are ladders or have at most 6 vertices, it is easily seen that they
both have at least two perfect matchings. We now distinguish three
cases (we omit symmetric cases) based on the number
$m^A_{34}$ (and $m^B_{56}$) of matchings in $G[A]$ ($G[B]$) covering
all the vertices but the end-vertices of $e_3$ and $e_4$ ($e_5$ and
$e_6$, respectively):
\begin{itemize}
\item {Let $m^A_{34}\ge 1$ and $m^B_{56}\ge 1$.}
      Remove all the vertices of $A\cup B$ and identify the edges $e_3$ and $e_4$
      to a single edge and the edges $e_5$ and $e_6$ to a single edge.
      Observe that the resulting graph is bridgeless and thus its expansion
      contains at least
      $$\tfrac{a+3}{24}\,(n-2\gamma b)-\beta_B=\tfrac{a+3}{24}\,n-\tfrac{a+3}{12}\,\gamma b-\beta_B$$
      perfect matchings by Lemma B.$a$. Each of these matchings can be extended
      to a perfect matching of $H$ avoiding $e$ and the bound follows.
\item {Let $m^A_{34}=0$ and $m^B_{56}=0$.} Observe that $G[A\cup B]$ contains a matching
      avoiding $e$ and covering all the vertices except the
      end-vertices of $e_3$ or $e_4$ (the edge can be prescribed) and
      $e_5$ or $e_6$ (again, the edge can be prescribed). To see this,
      observe that in $G^A_{(13)}$, there exists a perfect matching containing
      $e_3^A$. Since $m_{34}=0$, this matching also contains
      $e_2^A$. Similarly, considering perfect matchings of $G^A_{(14)}$ containing $e_4^A$ we get that $G[A]$ has a matching covering all the
      vertices except the end-vertices of $e_2$ and $e_4$; and the same
      holds for $G[B]$. The combination of these four matchings yields
      the desired result.
      
Remove now all the vertices of $A\cup B$, identify the end-vertices of
$e_3$ and $e_4$ and the end-vertices of $e_5$ and $e_6$ and add an
edge between the two new vertices.  Observe that the resulting graph
is bridgeless and thus its expansion contains at least
$$\tfrac{a+3}{24}\,(n-2\gamma
b)-\beta_B=\tfrac{a+3}{24}\,n-\tfrac{a+3}{12}\,\gamma b-\beta_B$$
perfect matchings by Lemma B.$a$. Each of these matchings can be
extended to a perfect matching of $H$ avoiding $e$ and the bound
follows.
\item {Let $m^A_{34}\ge 1$ and $m^B_{56}=0$.} 
Recall that each of $G[A]$ and $G[B]$ is a
      ladder or has at most 6 vertices. Hence, each of them is either
      the exceptional graph of Figure~\ref{fig:exept} or
      bipartite. Hence, $h^A_\varnothing\ge 2$ and $h^B_\varnothing\ge 2$ and therefore there are at least four perfect matchings of $G[A\cup B]$ avoiding $e$. 
      
      Observe that in the exceptional graph, all the values
      $m_{ij}$ are at least one, so $G[B]$ is necessarily bipartite.
Two of the four corners (vertices of degree two) are white, and
      two are black. Moreover, there is a matching covering all the
      vertices except any pair of corners of distinct colors, and
      there are no matchings covering all the vertices except a pair
      of corners of the same color. Since $m^B_{56}=0$, the
      end-vertices of $e_5$ and $e_6$ have the same color. Hence,
      there exist a matching covering all the vertices of $G[B]$
      except $e_2$ and $e_5$ (resp. $e_2$ and $e_6$). Consider perfect matchings of $G^A_{(12)}$ containing $e^A_2$. By symmetry, we may assume there is a matching of $G[A]$ covering all its vertices except the end-vertices of $e_2$ and $e_3$. 
      
      Altogether, these matchings can be combined to matchings of $G[A\cup B]$ avoiding $e$ covering all its vertices except:
      
      \begin{itemize} 
      \item {\it the end-vertices of $e_3$ and $e_5$}, and 
      \item {\it the end-vertices of $e_3$ and $e_6$}, and
      \item {\it the end-vertices of $e_3$ and $e_4$:}
      such a matching is obtained by combining a perfect matching of
      $G[B]$ and a matching of $G[A]$ covering all the vertices except
      the end-vertices of $e_3$ and $e_4$ (which exists since
      $m^A_{34}\ge 1$.) 
      \end{itemize}

\begin{figure}[htbp]
\begin{center}
\includegraphics[scale=0.5]{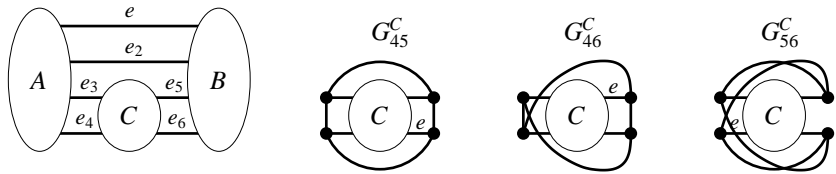}
\caption{When there are two edges between $A$ and $B$. \label{fig:eser2}}
\end{center}
\end{figure}

      Consider now the graphs
      $G^C_{ij}$, $\{i,j\}\subseteq\{4,5,6\}$ obtained from $G$ by
      removing all the vertices of $A\cup B$, introducing a new cycle
      of length four and making its vertices incident with the edges
      $e_i,e_3,e_j$ and the remaining edge which will play the role of
      $e$ (in this order). These three graphs are depicted in
      Figure~\ref{fig:eser2}, right. Applying Lemma E.$(a-1).b$ to the three
      graphs $G^C_{ij}$, we obtain the following
      inequalities: \begin{eqnarray*}
      h^C_{34}+h^C_{35}+2h^C_{\varnothing} & \ge &
      \tfrac{a+2}{24}\,(n-2\gamma b)-\beta_E \\
      h^C_{34}+h^C_{36}+2h^C_{\varnothing} & \ge &
      \tfrac{a+2}{24}\,(n-2\gamma b)-\beta_E \\
      h^C_{35}+h^C_{36}+2h^C_{\varnothing} & \ge &
      \tfrac{a+2}{24}\,(n-2\gamma b)-\beta_E \end{eqnarray*} where
      $h^C_X$ is the number of matchings of the expansion of $G[C]$ covering all its
      vertices except the end-vertices of the edges with indices from
      $X$. Observe that perfect matchings of
      $G^C_{ij}$ avoiding $e$ can be extended to perfect matchings of
      $H$ (avoiding the original $e$); those avoiding all the four edges incident with the cycle in at
      least four different ways. Finally, we obtain the following estimate on the number of
      perfect matchings of $H$ avoiding $e$: \begin{eqnarray*}
      h^C_{34}+h^C_{35}+h^C_{36}+4h^C_{\varnothing} & \ge &
      \tfrac32\cdot\left[\tfrac{a+2}{24}\,n - \tfrac{a+2}{12}\,\gamma
      b-\beta_E\right] \\ & \ge &
      \tfrac{a+3}{24}\,n-\tfrac{a+2}8\,\gamma
      b-\tfrac32\,\beta_E\mbox{.}  \end{eqnarray*}
\end{itemize}

It remains to consider the case that the edge $e$ is the only edge
between $A$ and $B$.  Let $e_2$, $e_3$ and $e_4$ be the three edges
between $A$ and $C$, and $e^\prime_2$, $e^\prime_3$ and $e^\prime_4$
the three edges between $B$ and $C$ (see Figure~\ref{fig:eser3},
left).  Recall that each of $G[A]$ and $G[B]$ is a ladder or has at
most six vertices. By symmetry, we can assume that, in addition to a
perfect matching, $G[A]$ contains a matching covering all its vertices
except the end-vertices of $e_2$ and one of the edges $e_3$ and $e_4$
(both choices possible). Symmetrically, for $G[B]$.  Remove now all
the vertices of $A\cup B$, identify the end-vertices of $e_3$ and
$e_4$ and join the new vertex to the end-vertex of
$e_2$. Symmetrically, for $e^\prime_2$, $e^\prime_3$ and $e^\prime_4$.
Finally, let $e$ be the edge joining the only two vertices of degree
two (see Figure~\ref{fig:eser3}, center). It can be verified that the
resulting graph $G'$ is cyclically $4$-edge-connected and $e$ is not in
any cyclic $4$-edge-cut of it unless $e$ is contained in a triangle in
$G'$.  Hence, unless $e$ is contained in a triangle in $G'$, by Lemma
D.$a.b$ the expansion of $G'$ has at least
$$\tfrac{a+3}{24}\,(n-2\gamma
b)-\beta_D=\tfrac{a+3}{24}\,n-\tfrac{a+3}{12}\,\gamma b-\beta_D$$
perfect matchings avoiding $e$ which all extend to the expansion of
$G$.

\begin{figure}[htbp]
\begin{center}
\includegraphics[scale=0.5]{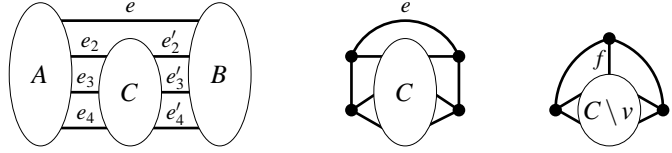}
\caption{When there is only one edge between $A$ and $B$. \label{fig:eser3}}
\end{center}
\end{figure}

Assume now that $e$ is contained in a triangle. In other words, the
edges $e_2$ and $e^\prime_2$ have a common vertex, say $v$, in $G$ and
let $f$ be the third edge incident with $v$.  Observe that $G'$ is
$2$-almost cyclically $4$-edge-connected (its only cyclic $3$-edge-cut
is the triangle containing $e$). Reduce the triangle (see
Figure~\ref{fig:eser3}, right) and apply Lemma E.$(a-1).b$.  Observe
that each matching of the expansion of the reduced graph avoiding $f$
can be extended in at least two different ways to a perfect matching
of $H$ avoiding $e$ (for any such matching, either none of the edges
of $E(A,B\cup C)$ is included and we use $h^A_\varnothing\ge 2$, or
none of the edges of $E(A\cup C,B)$ is included and we use
$h^B_\varnothing \ge 2$). Hence, the number of perfect matchings of
$H$ avoiding $e$ is at least
$$2\cdot\tfrac{a+2}{24}\,(n-2\gamma b-2b)-2\beta_E\ge
\tfrac{a+3}{24}\,n-\tfrac{a+2}{6}\,(\gamma+1)b-2\beta_E.$$
\end{proof}

This finishes the proof of the E-series of the lemmas and also
concludes the proof of Theorem~\ref{th:main}, which is readily seen to
be a direct consequence of the B-series. Note that from the E-series
we obtain the following result:

\begin{theorem}\label{th:c4c}
  For any $\alpha >0$ there exists a constant $\beta>0$ such that
  every $n$-vertex cyclically 4-edge-connected cubic graph has at
  least $\alpha n-\beta$ perfect matchings avoiding any given edge.
\end{theorem}

This does not hold for 3-edge-connected graphs: there exists an
infinite family of 3-edge-connected cubic graphs containing an edge avoided
by only two perfect matchings. However, recall that by Lemma~\ref{lm-3conn},
any 3-edge-connected cubic graph has a linear number of perfect
matchings avoiding any edge not contained in a cyclic 3-edge-cut.\\

Despite all our efforts, we were not able to replace the bound in
Theorem~\ref{th:main} by an explicit superlinear bound. We offer 1\,kg
of chocolate bars \emph{Studentsk\'a pe{\v c}et'} for the first
explicit bound derived from our proof. To get a superpolynomial or
even an exponential bound, one would probably like to insert
Lemma~\ref{lm-3conn} in the induction argument; we believe that the
linear bound in Lemma~\ref{lm-3conn} can be replaced by a bound
exponential in $n$.

\end{document}